\documentclass{article}
\usepackage[left=1in, right=1in, top=1in, bottom=1in]{geometry}
\usepackage{enumitem}
\usepackage{authblk}
\usepackage{graphicx}	
\usepackage{amsmath, amsthm, amssymb,mathtools}
\usepackage{xspace}
\usepackage{comment}
\usepackage{hyperref}
\usepackage[dvipsnames]{xcolor}
\usepackage{algorithm}
\usepackage{algpseudocode}
\usepackage{caption}

\numberwithin{equation}{section}

\newcommand\omitstuff[1]{}

\newcommand{\E}{\mathbb{E}}

\theoremstyle{plain}
\newtheorem{theorem}{Theorem}[section]
\newtheorem{proposition}[theorem]{Proposition}
\newtheorem{corollary}[theorem]{Corollary}
\newtheorem{lemma}[theorem]{Lemma}

\theoremstyle{definition}
\newtheorem{definition}[theorem]{Definition}
\newtheorem{remark}[theorem]{Remark}

\newtheorem{problem}[theorem]{Problem}
\newtheorem{ex}[theorem]{Example}
\newtheorem{observation}[theorem]{Observation}

\newcommand{\Lem}[1]{Lemma~\ref{#1}\xspace}
\newcommand{\Cor}[1]{Corollary~\ref{#1}\xspace}
\newcommand{\Prop}[1]{Proposition~\ref{#1}\xspace}
\newcommand{\Thm}[1]{Theorem~\ref{#1}\xspace}
\newcommand{\Def}[1]{Definition~\ref{#1}\xspace}

\newcommand{\wt}{\text{wt}}
\newcommand\T{\rule{0pt}{3ex}}

\DeclareMathOperator{\supp}{supp}

\DeclareMathOperator{\Inv}{Inv}
\DeclareMathOperator{\inv}{inv}
\DeclareMathOperator{\Des}{Des}
\DeclareMathOperator{\des}{des}

\DeclareMathOperator{\maj}{maj}

\DeclareMathOperator{\exc}{exc}
\DeclareMathOperator{\aexc}{aexc} 
\DeclareMathOperator{\pr}{Pr}
\DeclareMathOperator{\cDes}{cDes} 
\DeclareMathOperator{\cdes}{cdes} 
\DeclareMathOperator{\baj}{baj} 

\DeclareMathOperator{\cv}{cval} 
\DeclareMathOperator{\cpk}{cpk} 
\DeclareMathOperator{\CV}{Cval} 
\DeclareMathOperator{\CPK}{Cpk} 
\DeclareMathOperator{\cda}{cdasc} 
\DeclareMathOperator{\CDA}{Cdasc} 
\DeclareMathOperator{\cdd}{cddes} 
\DeclareMathOperator{\CDD}{Cddes} 

\title{Permutation Statistics in Conjugacy Classes\\of the Symmetric Group\footnote{This work was completed in part at the 2022 Graduate Research Workshop in Combinatorics, which was supported in part by NSF grant \#1953985 and a generous award from the Combinatorics Foundation. ML was partially supported by J. A. Grochow's NSF award CISE-2047756 and the University of Colorado Boulder, Department of Computer Science Summer Research Fellowship. MY was partially supported by the University of Denver's Professional Research Opportunities for Faculty Fund 80369-145601. We wish to thank Sara Billey for suggesting excedances and Yan Zhuang for bringing \cite{CooperJonesZhuang2020} to our attention.

We would also like to express our gratitude to Yan Zhuang for  kindly alerting us to the arXiv paper of Hamaker and Rhoades \cite{hamaker2022characters}, after seeing the first version of the present paper.  Finally we thank  Zach Hamaker for taking the time to  explain the results of the Hamaker--Rhoades paper and its overlap with the present work.
}}

\author[1]{Jesse Campion Loth}
\author[2]{Michael Levet}
\author[3]{Kevin Liu}
\author[4]{Eric Nathan Stucky}
\author[5]{Sheila~Sundaram}
\author[6]{Mei Yin} 

\affil[1]{Department of Mathematics, Simon Fraser University}
\affil[2]{Department of Computer Science, University of Colorado Boulder}
\affil[3]{Department of Mathematics, University of Washington}
\affil[4]{Department of Information Technology \& Sciences, Champlain College}
\affil[5]{Pierrepont School, Westport, CT, USA}
\affil[6]{Department of Mathematics, University of Denver}

\begin{document}
\maketitle

\begin{abstract}
We introduce the notion of a \emph{weighted inversion statistic} on the symmetric group, and examine its distribution on each conjugacy class. Our work generalizes the study of several common permutation statistics, including the number of inversions, the number of descents, the major index, and the number of excedances. As a consequence, we obtain explicit formulas for the first moments of several statistics by conjugacy class. We also show that when the cycle lengths are sufficiently large, the higher moments of arbitrary permutation statistics are independent of the conjugacy class. Fulman (\textit{J.~Comb.~Theory Ser.~A.}, 1998) previously established this result for major index and descents. We obtain these results, in part, by generalizing the techniques of Fulman (ibid.), and introducing the notion of \textit{permutation constraints}.  For permutation statistics that can be realized via \emph{symmetric} constraints, we show that each moment is a polynomial in the degree of the symmetric group.  
\end{abstract}



\thispagestyle{empty}

\newpage

\setcounter{page}{1}

\section{Introduction}

Let $S_n$ denote the symmetric group of permutations on $[n]=\{1,2,\dots,n\}$.  A  statistic on $S_n$ is a map $X : S_{n} \to \mathbb{R}$.  The \emph{distribution} of $X$ on $S_n$ is the function $(x_{k})_{k \in \mathbb{R}}$, where $x_k$ is mapped to the number of permutations $\omega\in S_n$ such that $X(\omega)=k$, i.e., $x_k=|X^{-1}(k)|$.  Perhaps the best known statistics are the numbers of descents, the major index, and the inversion number of a permutation (see \cite{StanEC1, StanEC2}).

We study the distributions of statistics on fixed conjugacy classes of $S_n$.  These distributions are known exactly for some classical statistics: Gessel and Reutenauer \cite[Theorems 5.3, 5.5, 6.1]{GesselReutenauer} gave a generating function for the joint distribution of descents and major index by conjugacy class.  Brenti \cite{Brenti1993} gave the generating function by conjugacy class for the excedance statistic in terms of the Eulerian polynomials.  Some asymptotic results are also known: Fulman \cite{FulmanJCTA1998} showed that descents and major index exhibit an asymptotically normal distribution on conjugacy classes with sufficiently large cycles.  Kim and Lee \cite{KimLee2020} subsequently extended this result to any conjugacy class of $S_n$.

We focus on the properties of the moments of these distributions. Fulman \cite{FulmanJCTA1998} showed that for partitions $\lambda \vdash n$ with each $\lambda_{i}>2\ell$, the $\ell$th moment for descents of the conjugacy class $C_{\lambda}$ is the same as for the entire symmetric group. In particular, this implies that the moments for descents and major index on a conjugacy class $C_{\lambda}$ are dependent only on the smaller part sizes of $\lambda$.  Fulman provided two proofs of this -- one using generating functions and the other a purely combinatorial proof that leveraged the structure of descent sets.  This paper will establish similar dependence results for all permutation statistics, not just those with special descent structure.

Inspired by the  combinatorial proof of \cite[Theorem~3]{FulmanJCTA1998}, we define a framework that allows us to calculate the first moment for multiple families of permutation statistics.  It turns out that the first moment for all these statistics is only dependent on the number of parts of size one and two in $\lambda$. The higher moments of these statistics are, in general, difficult to calculate explicitly. Remarkably, this framework allows us to show that the higher moments of all permutation statistics depend only on the small part sizes of $\lambda$.

Finally, we show that for a natural class of permutation statistics (see \Thm{thm:PropHigherMomentPolynomial}) that include inversions, permutation patterns, and excedances, these moments are polynomial in $n$. Using these polynomiality results and data for small values of $n$, we can explicitly calculate some higher moments of some permutation statistics. Gatez and Pierson \cite{GaetzPierson} established the analogous result for a different generalization of permutation patterns. While our generalization and that of Gaetz and Pierson \cite{GaetzPierson} agree for permutation patterns on certain conjugacy classes, it is not clear that they both capture the same family of permutation statistics. 

\vskip.1truein

\noindent \textbf{Main results.} In this paper, we study the uniform distribution of various permutation statistics on individual conjugacy classes. Our analysis of the uniform distribution of a very large class of permutation statistics is accomplished by the introduction of two notions: \textit{weighted inversion statistics} (Section~\ref{sec:weighted}) and \textit{(symmetric) permutation constraints} (Section~\ref{sec:highermoments}) on $S_{n}$. In fact, the classically defined inversions, descents, and major index are specific instances of weighted inversion statistics. While the notion of a weighted inversion statistic is new, the notion of a permutation constraint can be traced back to \cite[Theorem~3]{FulmanJCTA1998}. The notion of a permutation constraint is quite powerful, allowing us to reason about arbitrary permutation statistics. Although  symmetric constraints do not appear to include all weighted inversion statistics, they are still quite general, capturing inversions, permutation pattern statistics, and excedances.

We first examine the expected values of weighted inversion statistics on individual conjugacy classes, obtaining the following independence result.

\begin{theorem} \label{thm:Independence}
Let $\lambda=(1^{a_1},2^{a_2},\ldots,n^{a_n})\vdash n$. The expected value of any weighted inversion statistic in the conjugacy class $C_{\lambda}$  indexed by $\lambda$ depends only on $n$, $a_1$, and $a_2$. 
\end{theorem}

In the process of proving \Thm{thm:Independence}, we are able to derive explicit formulas for the expected values for several permutation statistics in individual conjugacy classes. See Table~\ref{tab: SummaryTable} for a summary of  our results, as well as a comparison to the first moments of these statistics on the entire symmetric group.

\begin{table}[h!]
\begin{center}
\begin{tabular}{|c|c|c|c|c||c|l|}
\hline
statistic & $\lambda=(1^{a_1}2^{a_2}\ldots)\vdash n$ & $\lambda_i\ge 3\ \forall i$ &$\lambda=(1^{a_1}2^{a_2})$  & $\lambda=(2^{a_2})$\T & All of $S_{n}$\\
[2pt]\hline
$\des$ & $\frac{n^2-n+2a_2-a_1^2+a_1}{2n}$   &$\frac{n-1}{2}$&  $\frac{n^2-a_1^2}{2n}$
&$\frac{n}{2}$ & $\frac{n-1}{2}$  \T \\[3pt]\hline
$\maj$ & $\frac{n^2-n+2a_2-a_1^2+a_1}{4}$    &$\frac{n(n-1)}{4}$ & $\frac{n^2-a_1^2}{4} $
&$\frac{n^2}{4}$ & $\frac{n^{2}-n}{4}$ \T \\[3pt]\hline
$\inv$ & $\frac{3n^2-n+2a_2-a_1^2+a_1 -2na_1}{12} $  &$ \frac{n(3n-1)}{12}  $ & $\frac{(3n+a_1)(n-a_1)}{12}   $ & $\frac{n^2}{4}$ & $\frac{n^{2}-n}{4}$  \T \\[3pt]\hline
$\baj$ & $\frac{(n+1)(n^2-n+2a_2-a_1^2+a_1)}{12}$ &$\frac{n(n^2-1)}{12}$ 
& $\frac{(n+1)(n^2-a_1^2)}{12}$ &$\frac{ n^2(n+1)}{12}$  & $\frac{1}{4} \binom{n+1}{3}$ \\[3pt]\hline
$\baj-\inv$ & $\frac{(n-2)(n^2-n+2a_2-a_1^2+a_1)}{12}$ &$\frac{n(n-1)(n-2)}{12}$ 
& $\frac{(n-2)(n^2-a_1^2)}{12}$ & $\frac{n^2(n-2)}{12}$  & $\frac{1}{4} \binom{n}{3}$  \T \\[3pt] \hline
$\cdes$ &  $\frac{n^2-n+2a_2-a_1^2+3a_1-2}{2(n-1)}$  & $ \frac{(n+1)(n-2)}{2(n-1)}$ 
 & $\frac{n^2-a_1^2+2a_1-2}{2(n-1)}$&$\frac{n^2-2}{ 2(n-1)  }$  & $\frac{n}{2}$  \T \\[3pt]\hline
$\widetilde{\exc}$ &$\frac{n+a_1}{2}$  &$\frac{n}{2}$ & $\frac{n+a_1}{2}=a_1+a_2$   &  $\frac{n}{2}=a_2$ & $\frac{n+1}{2}$  \T \\[3pt]\hline
$\exc, \aexc$ &$\frac{n-a_1}{2}$  &$\frac{n}{2}$ & $\frac{n-a_1}{2}=a_2$   &  $\frac{n}{2}=a_2$ & $\frac{n-1}{2}$  \T \\[3pt]\hline
$\cda,\cdd$ & $\frac{n-a_1-2a_2}{6}$ & $\frac{n}{6}$ & $0$ & $0$ & $\frac{n-2}{6}$   \\[3pt]
\hline
$\cv,\cpk$ & $\frac{n-a_1+a_2}{3}$ & $\frac{n}{3}$ & $\frac{n-a_1}{2}=a_2$ & $\frac{n}{2}=a_2$ & $\frac{2n-1}{6}$   \\[3pt]
\hline
\end{tabular}
\vskip .1in
\caption{\small Expected values of various statistics in the conjugacy class $C_\lambda$  and in $S_n$. }
\label{tab: SummaryTable}
\end{center}
\end{table}
\begin{remark}\label{rem:ExpWholeGroup} The generating function, expected value, and variance of $\des$ appear in Riordan \cite[p.~216]{Riordan}, while the generating function and expected value of $\inv$ are due to Rodrigues (\cite[p. 237]{Rodrigues1839}, \cite[Notes for Chapter 1]{StanEC1}).

The Mahonian statistics maj and inv are equidistributed over $S_n$ by MacMahon \cite{MacMahon1916}, with a bijective proof via Foata's second fundamental transformation \cite{FoataPAMS1968}. 

The Eulerian statistics exc and des are equidistributed over $S_n$ \cite{MacMahon1915}, \cite[Proposition~1.4.3]{StanEC1} with a bijective proof via the first fundamental transformation \cite{Riordan, Foata1970TheorieGD, StanEC1}. 
\end{remark}

When considering conjugacy classes where all cycles have length at least $3$, we generalize the combinatorial algorithm of Fulman \cite[Theorem 3]{FulmanJCTA1998}. Precisely, we consider the notion of a \textit{permutation constraint}, which allows us to specify values of a permutation for certain elements of the domain. We then analyze the structure of the corresponding directed graph (see Section \ref{sec:highermoments}). Remarkably, the notion of permutation constraint allows us to reason about arbitrary permutation statistics.

We now turn our attention to the higher moments of arbitrary permutation statistics. For a permutation statistic $X$ and a partition $\lambda \vdash n$, denote $\E_{\lambda}[X]$ to be the expected value of $X$ taken over the conjugacy class $S_{n}$ indexed by $\lambda$.

\begin{theorem} \label{thm:HigherMoments}
Let $X$ be a permutation statistic that is realizable over a constraint set of size $m$, and let $k \geq 1$. If $\lambda \vdash n$ has all parts of size at least $mk+1$, then $\E_{\lambda}[X^k]$ is independent of $\lambda$.
\end{theorem}

\begin{remark}
As descents are weighted permutation statistics of size $2$, our results in Table~\ref{tab: SummaryTable} and \Thm{thm:HigherMoments} imply \cite[Theorem~2]{FulmanJCTA1998} as a corollary.
\end{remark}

In Section \ref{sec:highermoments} we consider the class of permutation statistics realizable over symmetric constraint sets. Starting with a single symmetric constraint statistic on $S_{n_0}$, one can construct its \emph{symmetric extensions}  to $S_n$ with $n\geq 1$. This class of permutation statistics is quite broad -- including a number of well-studied statistics such as $\widetilde{\exc}, \exc, \aexc$ which have size $1$; $\inv, \cda, \cdd, \cv,\cpk$ which have size $2$; and $\mathrm{ile}$ which has size $\leq 3$. For a full account of these statistics, see Sections~\ref{sec:weighted}, \ref{sec:cyclic}, and \ref{sec:highermoments}, as well as \cite{blitvic2021permutations}.

\begin{theorem} \label{prop:PolyBoundGeneralizedStatistics}
Fix $k, m \geq 1$. Let $(\lambda_{n})$ be a sequence of partitions, where $\lambda_{n} \vdash n$ and all parts of $\lambda_{n}$ have size at least $mk+1$. Let $(X_{n})$ be a symmetric extension of a symmetric permutation statistic $X=X_{n_0}$ induced by a constraint set of size $m$. There exists a polynomial $p_X(n)$ depending only on $X$ such that $p_X(n) = \E_{\lambda_{n}}[X_{n}^{k}]$.
\end{theorem}

\begin{remark}
In the proof of \Thm{prop:PolyBoundGeneralizedStatistics} (see \Thm{thm:PropHigherMomentPolynomial}), we are able to control both the degree and leading coefficient of these polynomials.
\end{remark}

\begin{remark} \label{rmk:PermutationPatterns}
After proving \Thm{prop:PolyBoundGeneralizedStatistics}, we came across a result for permutation patterns due to Gaetz and Pierson \cite[Theorem~1.2]{GaetzPierson}, who generalized a previous result of Gaetz and Ryba \cite[Theorem~1.1(a)]{GaetzRyba}. While Gaetz and Ryba utilized partition algebras and character polynomials to obtain their result, the proof technique employed by Gaetz and Pierson was purely combinatorial. In particular, the method of  Gaetz and Pierson is quite similar to our techniques for establishing \Thm{prop:PolyBoundGeneralizedStatistics}. 

We show in Section \ref{sec:highermoments} that permutation pattern statistics (in which we track the number of occurrences of a given permutation pattern within a specified permutation) are a special case of symmetric permutation constraint statistics -- in fact, for infinitely many $m$, there exists a permutation pattern that can be realized by a symmetric constraint set of size $m$ -- but the latter is a more general class of statistics. Permutation patterns require that the constraints induce permutations on the occurrences of the pattern. For instance, an occurrence of the $213$-pattern in the permutation $\omega$ is a triple $x, y, z$ that occurs in the order $x \cdots y \cdots z$, with $y < x < z$. 

Our more general symmetric permutation constraint statistics, however, need not induce sub-permutations. For instance, we are able to specify triples $x, y, z$ such that $y < x < z$ and $y$ appears before both $x$ and $z$, without specifying the relative ordering of $x$ and $z$. With this in mind, a comparison of \Thm{prop:PolyBoundGeneralizedStatistics} and \cite[Theorem~1.2]{GaetzPierson} shows that these two results agree on permutation pattern statistics for conjugacy classes $C_{\lambda}$ where all parts have sufficiently large size.
\end{remark}

\begin{remark}
\Thm{prop:PolyBoundGeneralizedStatistics} has practical value in explicitly computing higher moments for individual conjugacy classes. Namely, if we compute $\E_{(n)}[X^k]$ for the class of $n$-cycles in $S_n$, taken over $\deg(\E_{(n)}[X^k])+1$ terms starting from $n = mk+1$, then we can use polynomial interpolation to obtain a closed form solution for $\E_{(n)}[X^k]$. Moreover, in light of \Thm{thm:HigherMoments}, this moment for full cycles is identical to $\E_{\lambda}[X^k]$, provided all parts of $\lambda$ are at least $mk+1$.
\end{remark}

\noindent \textbf{Further related work.} There has been considerable work on constructing generating functions for permutation statistics.

It is well known, for instance, that the inversion and major index statistics admit the same distribution on the entire symmetric group, with the  $q$-factorial as the generating function. Permutations with the $q$-factorial as their generating function are called \textit{Mahonian}. A general account of Mahonian statistics can be found here \cite{FoataMahonian}. It is known that Mahonian statistics are asymptotically normal with mean $\binom{n}{2}/2$ and variance $[n(n-1)(2n+5)]/72$ \cite{FoataMahonian}.

For a permutation $\omega$, let $\text{Des}(\omega)$ be the set of descents in $\omega$ (that is, the set of indices $i$ such that $\omega(i) > \omega(i+1)$). Let  $d(\omega) := |\text{Des}(\omega)| + 1$. The Eulerian polynomials as defined in \cite{StanEC1} serve as the generating functions for $d(\omega)$ (see \cite{MacMahon, Riordan}). See \cite{Foata1970TheorieGD} for a detailed treatment of the properties of Eulerian polynomials. It is known that $d(\omega)$ is asymptotically normally distributed on $S_{n}$, with mean $(n+1)/2$ and variance $(n-1)/12$ under the condition that the number of $i$-cycles vanishes asymptotically for all $i$ (an early reference is \cite[p.~216]{Riordan}; see also Fulman \cite{FulmanJCTA1998}, who in turn cites unpublished notes of Diaconis and Pitman \cite{diaconis1986permutations}). We note that descents also have connections to sorting and the theory of runs in permutations \cite[Section~5]{KnuthTAOCP3}, as well as to models of card shuffling \cite{DiaconisPersiGrath, BayerDiaconis, DiaconisGraham}.

\vskip.1truein

\noindent \textbf{Outline of paper.} We start in Section \ref{sec:preliminaries} by outlining necessary definitions and notation. In Section \ref{sec:Warmup}, we establish some results on the first moments of descents and major index that demonstrate some of the techniques that we apply in conjugacy classes of the symmetric group. In Sections \ref{sec:weighted} and \ref{sec:cyclic}, we establish results on first moments in conjugacy classes of the symmetric group, including Theorem \ref{thm:Independence} and Table \ref{tab: SummaryTable}. We then apply these results to the entire symmetric group in Section \ref{sec:S_n}. We conclude in Section \ref{sec:highermoments} by defining 
 permutation constraint statistics and establishing general results on their moments in conjugacy classes.

\section{Preliminaries}\label{sec:preliminaries}

We outline some definitions and results that will be used throughout our work. We start with three well-known statistics.

\begin{definition}  Let $\omega$ be a permutation in the symmetric group $S_n$. 
\begin{enumerate}
\item A \emph{descent} of $\omega$ is an index $i\in [n-1]$, such that $\omega(i)>\omega(i+1).$ We write
\[
\Des(\omega) = \{ i : \omega(i) > \omega(i+1) \}
\]
for the set of descents. We write $\des(\omega) := |\Des(\omega)|$ for the number of descents of $\omega$. Following \cite{FulmanJCTA1998}, we also denote $d(\omega) := \des(\omega) + 1$.

\item The \emph{major index} $\maj(\omega)$ of $\omega$  is the sum of its descents:
\[\maj(\omega):=\sum_{i \in \Des(\omega)} i.\]

\item An \emph{inversion} of $\omega$ is a pair of indices $(i,j)$ such that $1\leq i<j\leq n$ and $\omega(i)>\omega(j)$. We write 
\[
\Inv(\omega) = \{ (i, j) : i < j, \text{ but } \omega(i) > \omega(j) \}
\]
for the set of inversions. 
The \emph{inversion number} $\inv(\omega) := |\Inv(\omega)|$ is the number of inversions of $\omega$. 

\end{enumerate}
\end{definition}

\noindent \\ Denote by $C_\lambda$ the conjugacy class of the symmetric group $S_n$ indexed by the integer partition $\lambda$ of $n$.  The following fact is well known,  e.g., \cite{StanEC1} (or \cite{DummitFoote}).

\begin{proposition} \label{prop:SizeConjugacyStabilizer} The order of the centralizer of an element of cycle type $\lambda$ is $z_\lambda=\prod_i i^{a_i} a_i!$, where $\lambda$ has $a_i$ parts equal to $i,$ $i\ge 1$. For $\lambda\vdash n$, the order of the conjugacy class $C_\lambda$ is thus $\frac{n!}{z_\lambda}$.
\end{proposition}

Throughout this paper, we will use $\pr_{S_n}$ and $\pr_{\lambda}$ to denote probabilities in $S_n$ and $C_{\lambda}$ (with respect to the uniform measure). We similarly use $\E_{S_n}$ and $\E_{\lambda}$ for expected values on the corresponding probability spaces.

\section{Warm-up: first moments of descents and major index}
\label{sec:Warmup}

Fulman \cite{FulmanJCTA1998} previously determined the expected number of descents for \textit{all} conjugacy classes of $S_{n}$ without restriction to cycle types. In this section, we give an elementary, bijective proof for the expected number of descents in conjugacy classes where each cycle has length at least $3$. While our result does not fully encompass that of Fulman, our technique of conjugating by an involution provides a much simpler bijective proof. Furthermore, we will employ this technique in subsequent sections (see Section \ref{sec:InversionPairIndicators}).

\begin{definition}
    Let $\lambda \vdash n$ have all parts of size at least 2.  Define:
    \begin{align*}
       & \tau_{i,j} : C_{\lambda} \rightarrow C_{\lambda} \\
        &\tau_{i,j} (\omega) = (i \, j) \omega (i \, j).
    \end{align*}
\end{definition}

\begin{lemma}\label{lem:conjugation}
    For any fixed $i,j \in [n]$ and $\lambda$, $\tau_{i,j}$ is an involution on $C_{\lambda}$.
\end{lemma}
\begin{proof}
    Since $C_\lambda$ is closed under conjugating by permutations, the map is certainly well defined.  Also, applying it twice to any $\omega$ gives $(i \, j)(i \, j) \omega (i \, j)(i \, j) = \omega$.
\end{proof}

Fulman previously established the following.

\begin{theorem}[{\cite[Theorem~2]{FulmanJCTA1998}}] \label{thm:RecallFulman2} For a partition $\lambda$ of $n$ with $n_i$ $i$-cycles, let $C_\lambda$ be the conjugacy class corresponding to $\lambda$.  Then 
\begin{enumerate}
\item $\E_{\lambda}[\des] = \frac{n-1}{2}+\frac{ n_2-\binom{n_1}{2}}{n}$;
\item Fix $k\ge 0,$ and assume all parts of $\lambda$ have size at least $2k+1$. Then the $k$th moments of $\des(\omega)$ over $C_\lambda$ and over the full symmetric group $S_n$ are equal, i.e. 
\[\E_{\lambda}[\des^k]=\E_{S_{n}}[\des^k].\]
\end{enumerate}
\end{theorem}

\begin{remark}
In \cite[Theorem~2]{FulmanJCTA1998}, Fulman considered $\des(\omega)$ for part (1) and $d(\omega) = \des(\omega) + 1$ for part (2). This differs with \Thm{thm:RecallFulman2}, where we consider $\des(\omega)$ in both parts (1) and (2).
\end{remark}

\noindent Lemma~\ref{lem:conjugation} gives the following simple proof of the following restricted case of \Thm{thm:RecallFulman2}(1). In fact, we will actually obtain the entirety of \Thm{thm:RecallFulman2}(1) using generalizations of this technique in Section~\ref{sec:weighted}.

\begin{observation}
Suppose that all part sizes of $\lambda$ are at least $3$. Then applying $\tau_{i, i+1}$ gives a bijection between permutations in $C_\lambda$ with a descent at position $i$, and those without.
\end{observation}

\begin{corollary} Let $\lambda \vdash n$ such that each $\lambda_{i} \geq 3$. We have that:
    \[
    \E_{\lambda}[\des] = \frac{n-1}{2}.
    \]
\end{corollary}

\begin{proof}
    The previous proposition gives us that the probability of having a descent at any position $i$ is $1/2$.  There are $n-1$ possible positions for a descent, so the result follows.
\end{proof}

\section{Weighted inversion statistics}\label{sec:weighted}

In this section, we consider \emph{weighted inversion statistics}, which contain descents, major index, and the usual inversions as special cases. We will give an explicit formula for the mean on $C_{\lambda}$ of the indicator function of $(i,j)$ being an inversion. We then use this to derive a general formula for the expected value of any weighted inversion statistic on $C_{\lambda}$. We start with definitions.

\begin{definition}
Let $\omega \in S_{n}$, and let $1\leq i < j \leq n$. Define $I_{i,j}$ to be the indicator function for an inversion at $(i,j)$, i.e., $I_{i,j}(\omega) = 1$ if $\omega(i) > \omega(j)$ and $I_{i,j}(\omega)=0$ otherwise. 

A \emph{weighted inversion statistic} in $S_n$ is any statistic that can be expressed in the form $\sum_{1\leq i<j\leq n} \wt(i,j) I_{i,j}$, where $\wt(i,j)\in \mathbb{R}$ for all $i,j$. 
\end{definition}

\begin{remark}
Observe that descents, major index, and inversions are three examples of weighted inversion statistics. These can respectively be expressed as $\des(\omega)=\sum_{i=1}^{n-1} I_{i,i+1}(\omega)$, $\maj(\omega)=\sum_{i=1}^{n-1} i\cdot I_{i,i+1}(\omega)$, and 
$\inv(\omega)=\sum_{1\leq i<j \leq n} I_{i,j}(\omega)$. In general, if $X=\sum_{1\leq i<j\leq n} \wt(i,j) I_{i,j}$ is a weighted inversion statistic, we can use linearity to express
\begin{equation}\label{inversiondecomposition}
    \E_{\lambda}[X]=\sum_{1\leq i<j\leq n} \wt(i,j) \E_{\lambda}[I_{i,j}]=\sum_{1\leq i<j\leq n} \wt(i,j) \pr_{\lambda}[I_{i,j}=1].
\end{equation}
Hence, if we can explicitly formulate $\E_{\lambda}[I_{i,j}]=\pr_{\lambda}[I_{i,j}=1]$, then we can calculate $\E_{\lambda}[X]$. This approach also allows us to obtain similar results for other permutation statistics, such as excedances and cyclic descents.
\end{remark}

\subsection{Inversion indicator functions} \label{sec:InversionPairIndicators}
 
In this subsection, we consider the expected value of $I_{i,j}$ in $C_{\lambda}$ for any $\lambda=(1^{a_1},2^{a_2},\ldots,n^{a_n})\vdash n$. Our main result will be an explicit formula in terms of $n$, $a_1$, $a_2$, and the difference $j-i-1$. Surprisingly, the expected value of $I_{i,j}$ depends on $a_1$ and $a_2$ but is independent of $a_3,\ldots,a_n$, and depends on $i$ and $j$ through their difference $j-i$ but not the actual values of $i$ and $j$ themselves.

One of our main tools will be applying the map $\tau_{ij}$, as introduced in Section \ref{sec:Warmup}. Observe that for $\omega\in C_{\lambda}$,
\begin{equation*}\label{tauijw}
    \tau_{ij} ( \omega ) (i)=\begin{cases} \omega(j) & \text{ if $\omega(j)\notin \{i,j\}$} \\ j & \text{ if $\omega(j)=i$}  \\
i & \text{ if $\omega(j)=j$}
\end{cases} \qquad \tau_{ij} ( \omega ) (j) = \begin{cases} \omega(i) & \text{ if $\omega(i)\notin \{i,j\}$} \\
i & \text{ if $\omega(i)=j$} \\
j & \text{ if $\omega(i)=i$.}\end{cases}
\end{equation*}

\noindent Motivated by the above cases, we partition $C_{\lambda}$ into five sets based on $i$ and $j$:
\begin{equation}\label{omegasets}
    \begin{split}
        \Omega_1^{ij} & =\{\omega\in C_{\lambda}:\omega(i),\omega(j)\notin \{i,j\}\},\\
        \Omega_2^{ij} & =\{\omega \in C_{\lambda}:\omega(i)=j,\omega(j)=i\},\\
        \Omega_3^{ij} & =\{\omega\in C_{\lambda}:\omega(i)=i,\omega(j)=j\},\\
        \Omega_4^{ij} & =\{\omega\in C_{\lambda}:\omega(i)=j,\omega(j)\neq i\}\cup \{\omega\in C_{\lambda}:\omega(i)\neq j,\omega(j)= i\},\\
        \Omega_5^{ij} & =\{\omega\in C_{\lambda}:\omega(i)=i,\omega(j)\neq j\}\cup \{\omega\in C_{\lambda}:\omega(i)\neq i,\omega(j)=j\}.
    \end{split}
\end{equation}
Using the Law of Total Probability, we can decompose
\begin{equation}\label{totalprob}
    \begin{split}
        \pr_{\lambda}[I_{i,j} =1]=\sum_{k=1}^5 \pr_{\lambda}[\omega\in \Omega_k^{ij}]\cdot \pr_{\lambda}[I_{i,j}(\omega) =1 \mid \omega\in \Omega_k^{ij}].
    \end{split}
\end{equation}
We can explicitly compute the quantities in this sum.

\begin{lemma}\label{partition}
Let $\lambda=(1^{a_1},2^{a_2},\ldots,n^{a_n})\vdash n$, fix $i<j$ in $[n]$, and define $\Omega_k=\Omega_k^{ij}$ as in (\ref{omegasets}). Then
    \begin{enumerate}
        \item $\pr_{\lambda}[\omega\in \Omega_2]=\frac{2a_2}{n(n-1)},$
        \item $\pr_{\lambda}[\omega\in \Omega_3]=\frac{a_1(a_1-1)}{n(n-1)},$
        \item $\pr_{\lambda}[\omega\in \Omega_4]=\frac{2}{n-1}\cdot \left(1-\frac{a_1}{n}-\frac{2a_2}{n}\right),$ and
        \item $\pr_{\lambda}[\omega\in \Omega_5]=\frac{2a_1}{n}\cdot \left(1-\frac{a_1-1}{n-1}\right).$
    \end{enumerate}
\end{lemma}

\begin{proof}
We proceed as follows.
\begin{enumerate}
\item We first note that if $a_2=0$, then $\omega$ has no $2$-cycles. As $\Omega_{2}^{ij}$ is precisely the set of permutations of $C_{\lambda}$ containing the $2$-cycle $(ij)$, we have that $\pr_{\lambda}[\omega\in \Omega_2]=0$, which agrees with the formula given. 

If instead $a_2>0$, then $(ij)$ forming a cycle implies that the remaining $n-2$ elements have cycle type $(1^{a_1},2^{a_2-1},\ldots,n^{a_n})$. Then the probability that $(ij)$ forms a $2$-cycle is given by:
\[\frac{|C_{(1^{a_1},2^{a_2-1},\ldots,n^{a_n})}|}{|C_{(1^{a_1},2^{a_2},\ldots,n^{a_n})}|}=\frac{2a_2}{n(n-1)},\]

\noindent recalling that the formulas for the centralizer sizes are given by \Prop{prop:SizeConjugacyStabilizer}.

\item By definition, $\Omega_{3}^{ij}$ contains the permutations of $C_{\lambda}$ with fixed points at positions $i$ and $j$. Thus, if $a_1\in \{0,1\}$, then $\pr_{\lambda}[\omega\in \Omega_3]=0$, which agrees with the formula given. 

If instead $a_1>1$, then the probability that $(i)$ and $(j)$ form $1$-cycles is given by
\[\frac{|C_{(1^{a_1-2},2^{a_2},\ldots,n^{a_n})}|}{|C_{(1^{a_1},2^{a_2},\ldots,n^{a_n})}|}=\frac{a_1(a_1-1)}{n(n-1)}.\]

\item We first consider $\{\omega\in C_{\lambda}:\omega(i)=j,\omega(j)\neq i\}$. Using the Law of Total Probability, we decompose
$\pr_{\lambda}[\omega(i)=j,\omega(j)\neq i]$ into the sum of the following terms:
\[\pr_{\lambda}[i \text{ is in a 1 cycle of }\omega]\cdot \pr_{\lambda}[\omega(i)=j,\omega(j)\neq i| i \text{ is in a 1 cycle of }\omega],\]
\[\pr_{\lambda}[i \text{ is in a 2 cycle of }\omega]\cdot \pr_{\lambda}[\omega(i)=j,\omega(j)\neq i| i \text{ is in a 2 cycle of }\omega],\]
\[\pr_{\lambda}[i \text{ is not in a 1 or 2 cycle of }\omega]\cdot \pr_{\lambda}[\omega(i)=j,\omega(j)\neq i|i \text{ is not in a 1 or 2 cycle of }\omega].\]

The first two terms are 0, and hence we need only compute the third term. Observe that
\[\pr_{\lambda}[i \text{ is in a 1 cycle of }\omega]=\frac{|C_{(1^{a_1-1},2^{a_2},\ldots,n^{a_n})}|}{|C_{(1^{a_1},2^{a_2},\ldots,n^{a_n})}|}=\frac{a_1}{n}.\]
Using our result from (1),
\[\pr_{\lambda}[i \text{ is in a 2 cycle of }\omega]=\sum_{k\neq i} \pr_{\lambda}[\omega(i)=k,\omega(k)=i]= \frac{2a_2}{n}.\]
Hence, $\pr_{\lambda}[i \text{ is not in a 1 or 2 cycle of }\omega]=1-\frac{a_1}{n}-\frac{2a_2}{n}$.

Finally, consider conjugation by $\rho=(i)(1,2,\ldots,i-1,i+1,\ldots,n)$ on the elements in $\Omega_4$.
Since  $\rho$  acts by replacing each element of a  cycle by its image under $\rho$, it induces bijections among the sets 
\begin{center}$\{\omega\in C_{\lambda}: \omega(i)=k, i \text{ is not in a $1$ or $2$ cycle of }\omega\}$\end{center} for $k\in [n]\setminus \{i\}$. Hence, $\{\omega\in C_{\lambda}:i \text{ is not in a 1 or 2 cycle of }\omega\}$ decomposes into $n-1$ sets of the same size based on the image of $i$. 

We conclude that \[\pr_{\lambda}[\omega(i)=j,\omega(j)\neq i|i \text{ is not in a 1 or 2 cycle of }\omega]=\frac{1}{n-1}.\]
Combined, we have that
\[\pr_{\lambda}[\omega(i)=j,\omega(j)\neq i]=\frac{1}{n-1}\cdot \left(1-\frac{a_1}{n}-\frac{2a_2}{n}\right).\] Repeating this argument over $\{\omega \in C_{\lambda}:\omega(i)\neq j,\omega(j)= i\}$ and adding the two terms implies (3).

\item We similarly first consider $\{\omega\in C_{\lambda}:\omega(i)=i,\omega(j)\neq j\}$. Then
\begin{equation*}
    \begin{split}
        \pr_{\lambda}[\omega(i)=i,\omega(j)\neq j] & = \pr_{\lambda}[\omega(i)=i]\cdot \pr_{\lambda}[\omega(j)\neq j|\omega(i)=i] \\
        & = \frac{a_1}{n}\cdot \left(1-\pr_{\lambda}[\omega(j)=j|\omega(i)=i]\right)\\
        & = \frac{a_1}{n}\cdot \left(1-\frac{a_1-1}{n-1}\right).
    \end{split}
\end{equation*}
Repeating this argument over $\{\omega\in C_{\lambda}:\omega(i)\neq i,\omega(j)= j\}$ and adding this to the expression above implies the result. \qedhere 
\end{enumerate}
\end{proof}

\begin{remark}
The preceding lemma gives an explicit formula for $\pr_{\lambda}[\omega\in \Omega_1]$ using $1-\sum_{k=2}^5 \pr[\omega\in \Omega_k]$. We will not need this explicit formulation. 
\end{remark}

\begin{lemma}\label{conditionals}
Let $\lambda=(1^{a_1},2^{a_2},\ldots,n^{a_n})\vdash n$, fix $i<j$ in $[n]$, and define $\Omega_k=\Omega_k^{ij}$ as in (\ref{omegasets}). Then
\begin{enumerate}
    \item $\pr_{\lambda}[(i,j)\in \Inv(\omega)|\omega\in \Omega_1]=\frac12$,
    
    \item $\pr_{\lambda}[(i,j)\in \Inv(\omega)|\omega\in \Omega_2]=1$,
    
    \item $\pr_{\lambda}[(i,j)\in \Inv(\omega)|\omega\in \Omega_3]=0$, 
    
    \item $\pr_{\lambda}[(i,j)\in \Inv(\omega)|\omega\in \Omega_4]=\frac{1}{2}+\frac{j-i-1}{2(n-2)}$, and
    
    \item $\pr_{\lambda}[(i,j)\in \Inv(\omega)|\omega\in \Omega_5]=\frac{1}{2}-\frac{j-i-1}{2(n-2)}$.
\end{enumerate}
\end{lemma}

\noindent
\begin{remark}
A priori, it was not intuitively clear to us why:
\[
\pr_{\lambda}[(i,j)\in \Inv(\omega) \mid \omega\in \Omega_4] + \pr_{\lambda}[(i,j)\in \Inv(\omega) \mid \omega\in \Omega_5] = 1.
\]

\noindent Prior to proving \Lem{conditionals}, we first highlight our intuition here. If $k < i$ or $k > j$, then conjugating by $(ij)$ interchanges elements that have $(i,j)$ as an inversion to ones that do not. If $i < k < j$, then we have to track choices for $k$ and ``adjust" the probability from $1/2$. The $(j-i-1)/[2(n-2)]$ term accounts for this. Precisely, in $\Omega_{4}$, conjugating by $(ij)$ interchanges permutations that both have an inversion at $(i,j)$, and in $\Omega_{5}$, conjugating by $(ij)$ interchanges permutations that both do not have an inversion at $(i,j)$.
\end{remark}

\begin{proof}[Proof of Lemma~\ref{conditionals}]
\noindent
\begin{enumerate}
\item Note that the map $\tau_{ij}$ induces a bijection between the sets $\{\omega \in \Omega_1:\omega(i)>\omega(j)\}$ and $\{\omega\in \Omega_1:\omega(i)<\omega(j)\}$ that partition $\Omega_1$. Hence, these two sets must have the same size, and we conclude (1). 
\item This follows immediately from the definition of inversion and the images of $i$ and $j$ in the set $\Omega_2$. 

\item This follows immediately from the definition of inversion and the images of $i$ and $j$ in the set $\Omega_3$.

\item Observe that we can partition
\[ \{\omega\in C_{\lambda}:\omega(i)=j,\omega(j)\neq i\}=\bigsqcup _{k\notin\{ i,j\}} \{\omega\in C_{\lambda}: \omega(i)=j,\omega(j)=k\}.\] 
Now consider conjugation by
\[(i)(j)(1,2,\ldots,i-1,i+1,\ldots,j-1,j+1\ldots,n)\]
on $\Omega_4$. 
As in the proof of Lemma~\ref{partition}, this induces bijections among the sets $\{\omega\in C_{\lambda}: \omega(i)=j,\omega(j)=k\}$ for each $k \in [n] \setminus \{i,j\}$, and hence each of these disjoint sets has the same size. Additionally,  $\tau_{ij}$ induces a bijection between $\{\omega\in C_{\lambda}: \omega(i)=j,\omega(j)=k\}$ and $\{\omega\in C_{\lambda}: \omega(i)=k,\omega(j)=i\}$. Combining these two observations, we see that grouping elements by the images of $i$ and $j$ partitions $\Omega_4$ into $2(n-2)$ sets of the same size. Observe that the images of $i$ and $j$ are sufficient for determining if $(i,j)\in \Inv(w)$. When $\omega(i)=j$, $\omega(j)$ must be in $\{1,2,\ldots,j-1\}\setminus \{i\}$ to have an inversion at $(i,j)$. When $\omega(j)=i$, $\omega(i)$ must be in $\{i+1,\ldots,n\}\setminus \{j\}$ to have an inversion at $(i,j)$. Hence,
\begin{equation*}
    \begin{split}
        \pr_{\lambda}[(i,j)\in \Inv(\omega)|\omega\in \Omega_4] = \frac{(j-2)+(n-i-1)}{2(n-2)}= \frac{(n-2)+(j-i-1)}{2(n-2)} = \frac{1}{2}+\frac{j-i-1}{2(n-2)}.
    \end{split}
\end{equation*}

\item We can again partition $\Omega_5$ into $2(n-2)$ sets of the same size based on the image of $i$ and $j$. If $\omega(i)=i$, $\omega(j)$ must be in $\{1,2,\ldots,i-1\}$ to produce an inversion at $(i,j)$. If $\omega(j)=j$, then $\omega(i)$ must be in $\{j+1,\ldots,n\}$ to produce an inversion at $(i,j)$. Hence,
\begin{equation*}
\begin{split}
    \pr_{\lambda}[(i,j)\in \Inv(\omega)|\omega\in \Omega_5] & = \frac{(i-1)+(n-j)}{2(n-2)} 
     = \frac{(n-2)+(1+i-j)}{2(n-2)} 
    = \frac{1}{2}-\frac{j-i-1}{2(n-2)}. \qedhere
\end{split}
\end{equation*}
\end{enumerate}
\end{proof}

We have now established explicit formulas for all of the quantities in (\ref{omegasets}). Combining these, we compute the expected value of $I_{i,j}$ on $C_{\lambda}$. 

\begin{lemma}\label{Pij}
Let $\lambda=(1^{a_1},2^{a_2},\ldots,n^{a_n})\vdash n$. For any $i<j$ in $[n]$, 
\begin{equation*}
    \begin{split}
    \pr_{\lambda}[I_{i,j} =1]=\frac{1}{2}+ \frac{a_2}{n(n-1)}-\frac{a_1(a_1-1)}{2n(n-1)} + (j-i-1)\cdot \frac{n-n a_1-a_1+a_1^2-2a_2}{n(n-1)(n-2)}. 
    \end{split}
\end{equation*}
\end{lemma}

\begin{proof}
Define $\Omega_k=\Omega_k^{ij}$ as in (\ref{omegasets}). Starting with (\ref{totalprob}) and using Lemma \ref{conditionals}, $\pr_{\lambda}[I_{i,j} =1]$ can  be expressed as a sum of the following five terms: 
\begin{enumerate}[label=(\roman*)]
    \item $\pr_{\lambda}[\omega \in \Omega_1]\cdot \frac12$,
    \item $\pr_{\lambda}[\omega\in \Omega_2]\cdot \left(\frac12+\frac12\right)$,
    \item $\pr_{\lambda}[\omega\in \Omega_3]\cdot \left(\frac12-\frac12\right)$, 
    \item $\pr_{\lambda}[\omega\in \Omega_4]\cdot \left(\frac{1}{2}+\frac{j-i-1}{2(n-2)}\right)$, and
    \item $\pr_{\lambda}[\omega\in \Omega_5]\cdot \left(\frac{1}{2}-\frac{j-i-1}{2(n-2)}\right)$.
\end{enumerate}
We group terms with positive $1/2$ coefficients, use the fact that $C_{\lambda}$ is a disjoint union of $\{\Omega_k\}_{k=1}^5$, and apply Lemma \ref{partition} to obtain
\begin{equation*}
    \begin{split}
        & \frac{1}{2}\sum_{k=1}^5 \pr_{\lambda}[\omega\in \Omega_k]+\frac{1}{2} \pr_{\lambda}[\omega\in \Omega_2]-\frac{1}{2} \pr_{\lambda}[\omega\in \Omega_3] + \frac{j-i-1}{2(n-2)}\pr_{\lambda}[\omega\in \Omega_4] -\frac{j-i-1}{2(n-2)} \pr_{\lambda}[\omega\in \Omega_5] \\
        & = \frac{1}{2}+ \frac{a_2}{n(n-1)}-\frac{a_1(a_1-1)}{2n(n-1)} + \frac{j-i-1}{(n-1)(n-2)}\left(1-\frac{a_1}{n}-\frac{2a_2}{n}\right)-\frac{a_1(j-i-1)}{n(n-2)}\left(1-\frac{a_1-1}{n-1}\right). \\
        & = \frac{1}{2}+ \frac{a_2}{n(n-1)}-\frac{a_1(a_1-1)}{2n(n-1)} + (j-i-1)\cdot \frac{n-n a_1-a_1+a_1^2-2a_2}{n(n-1)(n-2)}.\qedhere
    \end{split}
\end{equation*}
\end{proof}

\subsection{First moment}

We now apply our results on $\E_{\lambda}[I_{i,j}]$ to calculate $\E_{\lambda}[X]$ for any weighted inversion statistic. We start with our main theorem on weighted inversion statistics.

\begin{theorem}\label{thm:generalinversion}
Let $\lambda=(1^{a_1},2^{a_2},\ldots,n^{a_n})\vdash n$, and let $X=\sum_{1\leq i<j\leq n} \wt(i,j) I_{i,j}$ be a weighted inversion statistic. Also set  $\alpha_n(X):=\sum_{1\leq i<j\leq n} \wt(i,j)$, 
and $\beta_n(X):=\sum_{1\leq i<j\leq n} (j-i-1) \wt(i,j)$. Then

\[
\E_{\lambda}[X]=\left(\frac{1}{2}+\frac{a_2}{n(n-1)}-\frac{a_1(a_1-1)}{2n(n-1)}\right)\cdot \alpha_n(X)+\left(\frac{n-n a_1-a_1+a_1^2-2a_2}{n(n-1)(n-2)}\right)\cdot \beta_n(X).
\]   
\end{theorem}

\begin{proof}  Note that $\alpha_n(X)$ and $\beta_n(X)$ are independent of the partition $\lambda$. 
We start with \eqref{inversiondecomposition} and apply Lemma \ref{Pij}  to see that $\E_{\lambda}[X]$ is given by
\begin{equation*}
    \begin{split}
        &\sum_{1\leq i<j\leq n} \wt(i,j) \pr_{\lambda}[I_{i,j}(\omega)=1] \\
        & = \sum_{1\leq i<j\leq n} \wt(i,j) \left(\frac{1}{2}+ \frac{a_2}{n(n-1)}-\frac{a_1(a_1-1)}{2n(n-1)} + (j-i-1)\cdot \frac{n-n a_1-a_1+a_1^2-2a_2}{n(n-1)(n-2)} \right) \\
        & = \left(\frac{1}{2}+\frac{a_2}{n(n-1)}-\frac{a_1(a_1-1)}{2n(n-1)}\right)\cdot \sum_{1\leq i<j\leq n} \wt(i,j)+\left(\frac{n-n a_1-a_1+a_1^2-2a_2}{n(n-1)(n-2)}\right)\cdot \sum_{1\leq i<j\leq n}\wt(i,j)(j-i-1). \qedhere
    \end{split}
\end{equation*}
\end{proof}

\begin{corollary}
    Let $\lambda=(1^{a_1},2^{a_2},\ldots,n^{a_n})\vdash n$. The expected value of any weighted inversion statistic in $S_n$ is independent of $a_3,\ldots,a_n$. 
\end{corollary}

We can apply the preceding theorem to obtain the expected number of some common statistics. Note that part (1) of the following corollary was previously established by Fulman \cite{FulmanJCTA1998}.

\begin{corollary}\label{cor:E-des-maj-inv}
    Let $\lambda=(1^{a_1},2^{a_2},\ldots,n^{a_n})\vdash n$, $n\ge 2$. Then
    \begin{enumerate}
        \item 
        $\E_{\lambda}[\des]=\frac{1}{2n}\left(n^2-n+2a_2-a_1^2+a_1\right)$,
        
        \item $\E_{\lambda}[\maj]=\frac{1}{4}\left(n^2-n+2a_2-a_1^2+a_1\right),$
        
        \item $\E_{\lambda}[\inv]=\frac{1}{12}\left(3n^2-n+2a_2-a_1^2+a_1 -2na_1\right).$
        
    \end{enumerate}
    In particular, in the case that $a_1=a_2=0$, we have that $\E_{\lambda}[\des]=\frac{n-1}{2}$, $\E_{\lambda}[\maj]=\frac{n(n-1)}{4}$, and $\E_{\lambda}[\inv]=\frac{3n^2-n}{12}$.
\end{corollary}

\begin{proof}
We use \Thm{thm:generalinversion} for all three statistics $X$. 
\begin{enumerate}
\item The descent statistic $\des $ is defined by $\wt(i,i+1)=1$ for $i\in \{1,2,\ldots,n-1\}$, and $\wt(i,j)=0$ otherwise. Hence $\alpha_n(X)=$ $\sum_{1\leq i<j\leq n} \wt(i,j)=(n-1)$ and  $\beta_n(X)=$ $\sum_{1\leq i<j\leq n}\wt(i,j)(j-i-1)=0$. Then
\[\E_{\lambda}[\des]=\left(\frac{1}{2}+\frac{a_2}{n(n-1)}-\frac{a_1(a_1-1)}{2n(n-1)}\right)\cdot(n-1)=\frac{1}{2n}{\left(n^2-n+2a_2-a_1^2+a_1\right)}.\]

\item The major index is defined by $\wt(i,i+1)=i$ and $\wt(i,j)=0$ otherwise. Now $\alpha_n(X)=$  $\sum_{1\leq i<j\leq n} \wt(i,j)={n\choose 2}$ and 
$\beta_n(X)=$ $\sum_{1\leq i<j\leq n}\wt(i,j)(j-i-1)=0$. Then
\[\E_{\lambda}[\maj]= \left(\frac{1}{2}+\frac{a_2}{n(n-1)}-\frac{a_1(a_1-1)}{2n(n-1)}\right)\cdot{n\choose 2}=\frac{1}{4}\left(n^2-n+2a_2-a_1^2+a_1\right).\]

\item Finally, the inversion statistic is defined by $\wt(i,j)=1$ for all $i,j$. Then $\alpha_n(X)=$
$\sum_{1\leq i<j\leq n} \wt(i,j)={n\choose 2}$, and using the substitution $k=j-i-1$, we find that 
$\beta_n(X)=$ $\sum_{1\leq i<j\leq n} \wt(i,j)(j-i-1)$ is given by
\begin{equation*}
    \begin{split}
        \sum_{1\leq i<j\leq n}(j-i-1)  =\sum_{i=1}^{n-1} \sum_{k=0}^{n-i-1} k
         =\sum_{i=1}^{n-1}\binom{n-i}{2}
        =\binom{n}{3}.
   \end{split}
\end{equation*}

Combined, we see that 
\begin{equation*}
    \begin{split}
        \E_{\lambda}[\inv] & = \left(\frac{1}{2}+\frac{a_2}{n(n-1)}-\frac{a_1(a_1-1)}{2n(n-1)}\right)\cdot {n\choose 2}+\left(\frac{n-n a_1-a_1+a_1^2-2a_2}{n(n-1)(n-2)}\right)\cdot {n\choose 3} \\
        & = \frac{1}{12}\left(3n^2-n+2a_2-a_1^2+a_1 -2na_1\right). \qedhere 
    \end{split}
\end{equation*}
\end{enumerate}
\end{proof}

\subsection{Baj}

In this subsection, we consider the curious permutation statistic $\baj$ that was introduced by Zabrocki \cite{Zabrocki2003}.

\begin{definition}[{\cite{Zabrocki2003}}]\label{defn:Zabrocki-baj}
Let $\omega \in S_{n}$. Define
\[\baj(\omega)\coloneqq \sum_{i\in \Des(\omega)} i(n-i).\]
\end{definition}

\noindent The statistic $\baj-\inv$ is the Coxeter length function restricted
to coset representatives of the extended affine Weyl group of type $A_{n-1}$ modulo translations by coroots. It has a nice generating function over the symmetric group, due to Stembridge and Waugh \cite{StembridgeWaugh1998}.  
 Furthermore, in \cite{BKS2020}, using this generating function, a formula for the $d$th cumulant is given \cite[Corollary~3.4]{BKS2020}, and  it is shown that the asymptotic distribution of  $\baj-\inv$ on $S_n$ is normal. 

Observe that $\baj$ is a weighted inversion statistic for the choice $\wt(i,i+1)=i(n-i)$ and $\wt(i,j)=0$ for $j\ne i+1$. Using Theorem~\ref{thm:generalinversion}, we obtain the following.

\begin{proposition}\label{prop:E-bajS} 
Let $\lambda=(1^{a_1},2^{a_2},\ldots,n^{a_n})\vdash n$, $n\ge 2$. Then 
\[
\E_\lambda[\baj]=\frac{1}{12}(n+1)(n^2-n+2a_2-a_1^2+a_1)=\frac{1}{3}(n+1)\E_\lambda[\maj ].
\]
\end{proposition}

\subsection{Cyclic descents}\label{sec:CycDes}

Cyclic descents were introduced by Paola Cellini \cite{Cellini1998}.  While these are not weighted inversion statistics, a small adjustment of the methods of the previous subsections allows us to compute the first moment of cyclic descents on $C_{\lambda}$.

\begin{definition}[{\cite{Cellini1998}}] The \emph{cyclic descent set} of a permutation $\omega\in S_n$ is defined to be the set \[\cDes(\omega):= \{1\le i\le n: \omega(i)>\omega(i+1) \subset [n]\},\]
with the convention $\omega(n+1):=\omega(1)$.  Let $\cdes(\omega):=|\cDes(\omega)|.$
\end{definition}

\begin{theorem}\label{thm:cDes-1st-momentS} Let $\lambda=(1^{a_1},2^{a_2},\ldots,n^{a_n})\vdash n$, $n\ge 2$. Then 
\[\E_\lambda[\cdes]=\frac{n}{2}+\frac{a_2-\binom{a_1}{2}}{n-1} +\frac{a_1-1}{n-1},\]
and hence the expected value of cyclic descents is independent of the conjugacy class if $a_1=a_2=0$. 
\end{theorem}
\begin{proof}
Writing  $J_n$ for the random variable which equals 1 if $n\in \cDes(\omega)$ and 0 otherwise,  we have 
\begin{equation}\label{eqn:EcDes1}
 \E_\lambda[\cdes]
       =\sum_{1\le i\le n-1} \pr_\lambda[I_{i,i+1}=1] + \pr_\lambda[J_n=1]\\
        =\E_\lambda[\des]+ \pr_\lambda[J_n=1].
\end{equation}
From Lemma~\ref{Pij} we have 
\begin{equation*}
\pr_\lambda[I_{1,n}=1]=\frac{1}{2} +\frac{a_2}{n(n-1)} -\frac{a_1(a_1-1)}{2n(n-1)} 
+ \frac{n-na_1-a_1+a_1^2-2a_2}{n(n-1)}
=\frac{1}{2} +\frac{\binom{a_1}{2} -a_2-n(a_1-1)}{n(n-1)}.
\end{equation*}
Now  $n$ is a cyclic descent if and only if $\omega(n)>\omega(1)$, i.e., if and only if $(1,n)$ is \emph{not} an inversion. Hence we have 
\begin{equation}\label{eqn:cDes1}
\begin{split}
\pr_\lambda[J_n=1] =1-\pr_\lambda[I_{1,n}=1]
                         = \frac{1}{2} +\frac{a_2-\binom{a_1}{2} +n(a_1-1)}{n(n-1)}.
\end{split}
\end{equation}

\noindent From Corollary~\ref{cor:E-des-maj-inv}, we have 
\begin{equation}\label{eqn:EDes}
\E_\lambda[\des]=\frac{n-1}{2} +\frac{a_2-\binom{a_1}{2}}{n}.
\end{equation}

\noindent Equation (\ref{eqn:EcDes1}) now gives the result.
\end{proof}

\section{Cyclic permutation statistics}\label{sec:cyclic}

In this section, we apply the techniques from Section~\ref{sec:weighted} to the cases of several other permutation statistics that are not weighted inversion statistics. Such permutation statistics include cyclic descents and excedances. We call these cyclic permutation statistics, to reflect the fact that, in general, the value of the statistic can be read directly from the cycles in its cycle decomposition. 

In particular, we show that, once again, the expected values depend on at most the number of fixed points and $2$-cycles in the cycle type.

\subsection{Excedances}\label{sec:Exc}

An \emph{excedance} of $\omega$ is any index $i\in [n]$ such that $\omega(i)>i$.
A \emph{weak excedance} of $\omega$ is any index $i\in [n]$ such that $\omega(i)\ge i$.  An \emph{anti-excedance} \cite{blitvic2021permutations} of $\omega$ is any index $i\in [n]$ such that $\omega(i)< i$.
Clearly $i$ is an excedance of $\omega$ if and only if $\omega(i)$ is an anti-excedance of $\omega^{-1}$, and conjugacy classes in $S_n$  are closed with respect to taking inverses, so for any fixed conjugacy class, excedance and anti-excedance are equidistributed.

Let $\exc(\omega)$ (respectively $\widetilde{\exc}(\omega), \aexc(\omega)$) denote the number of excedances (respectively weak excedances, anti-excedances) of the permutation $\omega$. While these are not weighted inversion statistics, the methods of Section \ref{sec:weighted} can be adapted to calculate their expected values in $C_{\lambda}$.

\begin{theorem}\label{thm:Elambda-exc}
Let $\lambda=(1^{a_1},2^{a_2},\ldots,n^{a_n})$. Then \[\E_{\lambda}[\exc]=\frac{1}{2}(n-a_1)=\E_\lambda[\aexc]  \text{ and }
\E_\lambda[\widetilde{\exc}]
          = \frac{1}{2}(n+a_1). \]
\end{theorem}

\begin{proof}
Express $\exc(\omega)=\sum_{j=1}^n I_{j}(\omega)$, where $I_j$ is the indicator random variable on an excedance at position $j$.  Fixing $j$,  partition $C_{\lambda}$ into the two sets $\Omega_1=\{w\in C_{\lambda}:\omega(j)=j\}$ and $\Omega_2=\{w\in C_{\lambda}:\omega(j)\neq j\}$. Then \[\pr_{\lambda}[I_j=1]=\pr_{\lambda}[\omega\in \Omega_1]\cdot \pr_{\lambda}[I_j(\omega)=1|\omega\in \Omega_1]+\pr_{\lambda}[\omega\in \Omega_2]\cdot \pr_{\lambda}[I_j(\omega)=1|\omega\in \Omega_2].\]
Observe that $\pr_{\lambda}[\omega\in \Omega_1]=\frac{a_1}{n}$ and $\pr_{\lambda}[I_j(\omega)=1|\omega\in \Omega_1]=0$. For $\pr_{\lambda}[I_j(\omega)=1|\omega\in \Omega_2]$, we can partition
\[\Omega_2=\bigsqcup_{k\neq j} \{w\in \Omega_2:w(j)=k\}.\]

Conjugation by $(j)(1,2,\ldots,j-1,j+1,\ldots,n)$ induces bijections among these sets, and thus they all must have the same size. Observe that in $n-j$ of the $n-1$ sets, an excedance at $j$ occurs. Hence,
\[\pr_{\lambda}[I_j=1]=\pr_{\lambda}[\omega\in \Omega_2]\cdot \pr_{\lambda}[I_j(\omega)=1|\omega\in \Omega_2]=\left(1-\frac{a_1}{n}\right)\cdot \frac{n-j}{n-1}.\]

For the excedance statistic, we conclude that
\begin{equation*}
    \begin{split}
        \E_{\lambda}[\exc]  = \sum_{j=1}^n \pr_{\lambda}[I_j=1]  = \sum_{j=1}^n \left(1-\frac{a_1}{n}\right)\cdot \frac{n-j}{n-1}  = \left(\frac{n-a_1}{n}\right)\cdot \frac{1}{n-1} \cdot {n\choose 2}  = \frac{1}{2}(n-a_1).
    \end{split}
\end{equation*}

\noindent We have already noted that for every fixed conjugacy class $C$, excedance and anti-excedance are equidistributed on $C$. 
For the weak excedance statistic $\mathrm{\widetilde{exc}}(\omega),$ by definition, the only change in the above argument is that 
$\pr_{\lambda}[\widetilde{I}_j(\omega)=1|\omega\in \Omega_1]=1$ where $\widetilde{I}_j$ is the weak excedance indicator function. 
Hence 
\begin{equation} \label{eqn:exc-vs-weakexcS}
\pr_{\lambda}[\widetilde{I_j}(\omega)=1]
=\pr_\lambda[I_j(\omega)=1]+\frac{a_1}{n},
\end{equation}

and \begin{equation*}
        \E_\lambda[\widetilde{\exc}]
        = \E_{\lambda}[\exc] + a_1
          = \frac{1}{2}(n+a_1). \qedhere
\end{equation*}

\end{proof}

\begin{corollary}
    Let $\lambda=(1^{a_1},2^{a_2},\ldots,n^{a_n})$. Then the expected values of $\exc, \widetilde{\exc}$  and $\aexc$ are independent of $a_2,\ldots,a_n$. In particular, when $a_1=0$, we have that $\E_{\lambda}[\exc]=\E_\lambda[\aexc]=\E_{\lambda}[\widetilde{\exc}]=\frac{n}{2}$.
\end{corollary}

\subsection{Cyclic double ascents and cyclic valleys}
Several recent papers \cite{CooperJonesZhuang2020,blitvic2021permutations} consider statistics derived from the excedance statistic.   In \cite{CooperJonesZhuang2020}, the following statistics are defined for $\omega\in S_n$.  The element $i\in [n]$ is a 
\begin{enumerate}
    \item  \emph{cyclic valley} of $\omega$ if $\omega^{-1}(i)>i<\omega(i)$;
    \item  \emph{cyclic peak} of $\omega$ if $\omega^{-1}(i)<i>\omega(i)$;
    \item \emph{cyclic double ascent} of $\omega$ if $\omega^{-1}(i)<i<\omega(i)$; and
    \item \emph{cyclic double descent} of $\omega$ if $\omega^{-1}(i)>i>\omega(i)$.
\end{enumerate}

A cyclic double ascent (respectively, cyclic double descent) coincides with the 
\emph{linked excedance} (respectively, \emph{linked anti-excedance}) defined in \cite{blitvic2021permutations}. We follow the notation of \cite{CooperJonesZhuang2020}, and write $\cv(\omega)$ (respectively, $\cpk(\omega)$) for the number of cyclic valleys (respectively, cyclic peaks) of $\omega$. Also write $\CV(\omega)$ (respectively, $\CPK(\omega)$) for the \emph{set}  of cyclic valleys (respectively, cyclic peaks) of $\omega$. Clearly $i$ is a cyclic valley of $\omega$  if either $i$ is the smaller  letter in a 2-cycle, or  if $i$ appears in a cycle of $\omega$ of length at least 3. In the latter case the cycle containing $i$ must be of the form $(\ldots j\,i\,k\ldots)$ for $j>i<k$.  Let $\rho$ be the reversing involution defined by $\rho(i)=n+1-i$. Since the corresponding cycle of $\rho\,\omega\rho^{-1}$ is $(\ldots, n+1-j,\,n+1-i,\,n+1-k,\ldots)$, it follows that 
\[ i\in \{1,\ldots, n-1\} \text{ is a cyclic valley of } \omega \iff n+1-i\in\{2, \ldots , n\} \text{ is a cyclic peak of } \rho\,\omega\rho^{-1},\]
and hence cyclic valleys and cyclic peaks are equidistributed over a fixed conjugacy class.
The same argument shows that cyclic double descents and cyclic double ascents are equidistributed over a fixed conjugacy class. 

The number of cyclic double ascents  (respectively cyclic double descents) in a permutation $\omega$ is denoted $\cda(\omega)$ (respectively,  $\cdd(\omega)$).  Also, the \emph{set} of cyclic double ascents  (respectively cyclic double descents) in a permutation $\omega$ is denoted $\CDA(\omega)$ (respectively,  $\CDD(\omega)$).

Now observe that our methods apply to the statistics $\cda(\omega)$, $\cv(\omega)$ and $\cdd(\omega)$, $\cpk(\omega)$ as well.
Let $I_j$ be the indicator function for a cyclic double ascent at index $j$ and decompose $\cda(\omega)=\sum_{j=2}^{n-1} I_j(\omega)$.  Let $I^v_j$ be the indicator function for a cyclic valley at $j$, and write $\cv(\omega)=\sum_{j=1}^{n-1} I^v_j(\omega)$.
 Define the sets 
\begin{equation}\label{shortcyclesets}
    \begin{split}
        \Omega_1^j & = \{\omega\in C_{\lambda}: \text{ $j$ is in a $1$-cycle}\}, \\
        \Omega_2^j & = \{\omega\in C_{\lambda}: \text{ $j$ is in a $2$-cycle}\},\\
        \Omega_3^j & = \{\omega\in C_{\lambda}: \text{ $j$ is not in a $1$-cycle or $2$-cycle}\}.
    \end{split}
\end{equation}
Similar arguments as before imply the following results.  First, we have the analogue of Lemma~\ref{partition}.

\begin{lemma}
Let $\lambda=(1^{a_1},2^{a_2},\ldots,n^{a_n})\vdash n$, fix $j\in [n]$, and define $\Omega_k=\Omega_k^j$ as in (\ref{shortcyclesets}). Then
\begin{enumerate}
    \item $\pr_{\lambda}[\omega\in \Omega_1]=\frac{a_1}{n}$,
    \item $\pr_{\lambda}[\omega\in \Omega_2]=\frac{2a_2}{n}$, and
    \item $\pr_{\lambda}[\omega\in \Omega_3]=1-\frac{a_1}{n}-\frac{2a_2}{n}$.
\end{enumerate}
\end{lemma}

\begin{proof}
    The proof follows the same arguments as Lemma~\ref{partition}.
\end{proof}

\begin{theorem}\label{thm:Elambda-cdasc-cval}
Let $\lambda=(1^{a_1},2^{a_2},\ldots,n^{a_n})\vdash n$. Then 
\begin{enumerate}
\item
$\E_{\lambda}[\cda]=\frac{n-a_1-2a_2}{6}=\E_{\lambda}[\cdd]$ and 
\item $\E_{\lambda}[\cv]=\frac{n-a_1+a_2}{3}=\E_{\lambda}[\cpk].$
\end{enumerate}
\end{theorem}

\begin{proof}
Fix $j$ and observe that if $\omega\in \Omega_1^j\cup \Omega_2^j$, then $j$ is not a cyclic double ascent of $\omega$. Also, 
 $j$ is a cyclic valley of $\omega$ only if  $\omega\in \Omega_2^j\cup \Omega_3^j $.
Hence, by the Law of Total Probability, 
we have 
\begin{align*}\E_{\lambda}[I_j]&=\sum_{k=1}^3 \pr_{\lambda}[\omega\in \Omega_k^j]\pr_{\lambda}[I_j(\omega)=1|\omega\in \Omega_k^j]=\pr_{\lambda}[\omega\in \Omega_3^j]\pr_{\lambda}[I_j(\omega)=1|\omega\in \Omega_3^j],\\
\E_{\lambda}[I^v_j]&=\sum_{k=1}^3 \pr_{\lambda}[\omega\in \Omega_k^j]\pr_{\lambda}[I^v_j(\omega)=1|\omega\in \Omega_k^j]\\&=\pr_{\lambda}[\omega\in \Omega_3^j]\pr_{\lambda}[I^v_j(\omega)=1|\omega\in \Omega_3^j]
+\pr_\lambda[\omega\in \Omega_2^j]
\pr_{\lambda}[I^v_j(\omega)=1|\omega\in \Omega_2^j].
\end{align*}
If we fix distinct $i,k\in [n] \setminus \{j\}$, then conjugation by appropriate elements implies $\Pr_{\lambda}[\omega(i)=j|\omega\in \Omega_2^j]=\Pr_{\lambda}[\omega(i)=j|\omega\in \Omega_3^j]=\frac{1}{n-1}$ and $\Pr_{\lambda}[\omega(i)=j\wedge \omega(j)=k|\omega\in \Omega_3^j]=\frac{1}{(n-1)(n-2)}$.

Now let $i,j,k$ be elements appearing in succession in a cycle of length at least 3. A cyclic double ascent at $j$ occurs if and only if $i<j<k$, and hence there are a total of $(j-1)(n-j)$ choices $\{i,k\}$ that result in a cyclic double ascent at $j\ne 1, n$. 
A cyclic valley occurs if  $i>j<k$, and thus there are a total of 
$(n-j)(n-j-1)$ choices for $\{i,k\}$ that result in a cyclic valley at $j\ne n$.  However, a cyclic valley also occurs at $j$ when  $(i,j)$ is a 2-cycle with $i>j$. There are $(n-j)$ choices for $i$ in this case.

Combined with the preceding lemma, we see that 
\begin{align*}
    \E_{\lambda}[I_j] & =\left(1-\frac{a_1}{n}-\frac{2a_2}{n}\right)\cdot \frac{(j-1)(n-j)}{(n-1)(n-2)},\\
    \E_{\lambda}[I^v_j] & =\left(1-\frac{a_1}{n}-\frac{2a_2}{n}\right)\cdot \frac{(n-j-1)(n-j)}{(n-1)(n-2)} + \frac{2a_2}{n}\cdot  \frac{n-j}{n-1}.
\end{align*}
Summing over all $j$ gives
\begin{align*}
    \E_{\lambda}[\cda]&=\left(1-\frac{a_1}{n}-\frac{2a_2}{n}\right)\cdot \frac{1}{(n-1)(n-2)}\cdot \sum_{j=2}^{n-1}(j-1)(n-j)=\left(1-\frac{a_1}{n}-\frac{2a_2}{n}\right)\cdot \frac{n}{6}, \\
    \E_{\lambda}[\cv]&=\left(1-\frac{a_1}{n}-\frac{2a_2}{n}\right)\cdot \frac{1}{(n-1)(n-2)}\cdot \sum_{j=1}^{n-1}(n-j-1)(n-j)+\frac{2a_2}{n(n-1)}\cdot  \sum_{j=1}^{n-1} (n-j)\\
&=\left(1-\frac{a_1}{n}-\frac{2a_2}{n}\right)\cdot \frac{n}{3} +a_2,
\end{align*}
using the facts that 
$\sum_{j=2}^{n-1} (j-1)(n-j)=\binom{n}{3}$ and $\sum_{j=1}^{n-1}(n-j-1)(n-j)=2\binom{n}{3}.$
This finishes the proof.
\end{proof}
These results confirm the fact that $\exc(\omega)=\cv(\omega)+\cda(\omega)$.%

\section{First moments on $S_{n}$ from conjugacy class}\label{sec:S_n}

In this section, we consider connections between the first moments on conjugacy classes with those on all of $S_n$. Observe that the expected values of a statistic $X$ on individual conjugacy classes is related to  the expected value on the entire symmetric group by the  formula
\begin{equation}\label{eqn:lambda-to-Sn}\E_{S_n}[X]=\sum_{\lambda\vdash n} z_\lambda^{-1} \E_\lambda[X], \end{equation}
since the order of the conjugacy class indexed by $\lambda$ is $n!/z_\lambda$.

In this section we  analyse Equation~\eqref{eqn:lambda-to-Sn} more carefully.  The following identities will be  useful.

\begin{lemma}\label{lem:zlambda-formulaS}
Let $\lambda=(1^{a_1},2^{a_2},\ldots,n^{a_n})\vdash n$. The following identities hold:
\begin{enumerate}
    \item $\sum_{\lambda\vdash n} z_\lambda^{-1}=1$,
    \item $\sum_{\lambda\vdash n} z_\lambda^{-1}a_1=1$,
    \item $\sum_{\lambda\vdash n} z_\lambda^{-1}a_1^2=2$, and
    \item $\sum_{\lambda\vdash n} z_\lambda^{-1}a_2=1/2.$
\end{enumerate}

\begin{proof} 
\noindent 
\begin{enumerate}
\item This is the class equation for $S_n$ \cite{DummitFoote}, a consequence of the fact that $n!=\sum_{\lambda\vdash n} |C_\lambda|.$

\item This is Burnside's lemma for the symmetric group \cite{DummitFoote,StanEC2}. 

\item Here we consider $S_n$ acting on 2-subsets of $[n]$.  There is only one orbit, and a permutation  fixes a 2-subset $\{i,j\}$ if and only if either $i,j$ are both fixed points, or $i,j$ form a 2-cycle. Hence the number of 2-subsets fixed by  a permutation of cycle type $\lambda$ with  $a_k$ parts of length $k$, is $\binom{a_1}{2}+a_2$, and Burnside's lemma gives 
\begin{equation}\label{eqn:Burnside2sets}
\sum_{\lambda\vdash n} z_\lambda^{-1} \left(\binom{a_1}{2}+a_2\right)=1.
\end{equation}
Similarly, by applying Burnside's lemma to the action of $S_n$ on the set $[n]\times [n]$ of ordered pairs $(i,j)$, which has two orbits $\{(i,i): 1\le i \le n\}$ and $\{(i,j):1\le i,j \le n, i\ne j\} $, and counting fixed points, we obtain
\begin{equation}\label{eqn:Burnside2pairs}
\sum_{\lambda\vdash n} z_\lambda^{-1} \left(a_1+2 \binom{a_1}{2}\right)=2
=\sum_{\lambda\vdash n} z_\lambda^{-1} a_1^2.
\end{equation}

\item Using~\eqref{eqn:Burnside2pairs} and the second identity also gives
\begin{equation}\label{eqn:Burnside3}
\sum_{\lambda\vdash n} 2 z_\lambda^{-1}  \binom{a_1}{2}=1.
\end{equation} The last identity now follows from~\eqref{eqn:Burnside2sets} and~\eqref{eqn:Burnside3}.\qedhere

\end{enumerate}
\end{proof}
\end{lemma}

It is now easy to compute the first moments of the preceding statistics over the whole symmetric group; see Table~\ref{tab: SummaryTable} for an overview of our results, as well as  a comparison to the literature. Note that we are able to obtain the first moment over the whole symmetric group without knowledge of the generating function for the statistic. 
Recall the definitions of $\alpha_n(X)=\sum_{1\leq i<j\leq n} \wt(i,j)$ and $\beta_n(X)=\sum_{1\leq i<j\leq n} (j-i-1) \wt(i,j)$ for a weighted inversion statistic $X=\sum_{1\leq i<j\leq n} \wt(i,j)I_{i,j}$ from Theorem~\ref{thm:generalinversion}.

\begin{proposition}\label{prop:Conjclass-to-Sn}  Let $\lambda=(1^{a_1},2^{a_2},\ldots,n^{a_n})\vdash n$, and let $X=\sum_{1\leq i<j\leq n} \wt(i,j) I_{i,j}$ be a weighted inversion statistic. Then
\begin{enumerate}
\item $\E_{S_n}[X]=\frac{\alpha_n(X)}{2}$, and
\item
$\E_\lambda[X]= \E_{S_n}[X] + f^{X}_n(a_1, a_2),$
where $f^{X}_n$ is a polynomial of degree at most $2$ in $a_1$ and $a_2$ such that 
\[\sum_{\lambda\vdash n} z_\lambda^{-1} f^{X}_n(a_1, a_2)=0.\]
\end{enumerate}
\end{proposition}
\begin{proof} 
Note first that  $\pr_{S_n}[I_{i,j}=1]=1/2$ for $1\le i<j\le n$. The decomposition $X=\sum_{1\le i<j\le n} \wt(i,j) I_{i,j}$ implies
\[\E_{S_n}[X]=\frac{1}{2}\sum_{1\leq i<j\leq n} \wt(i,j).\]
Although we can now conclude Part (2) as well, it is instructive to 
 examine the different contributions to our expression for $\E_\lambda[X]$  more carefully. Since $\beta_n(X)=\sum_{1\leq i<j\leq n} (j-i-1) \wt(i,j)$, from Theorem~\ref{thm:generalinversion} we obtain 
\begin{align*}
\E_{\lambda}[X]&=\left(\frac{1}{2}+\frac{a_2}{n(n-1)}-\frac{a_1(a_1-1)}{2n(n-1)}\right)\alpha_n(X)+\left(\frac{n-n a_1-a_1+a_1^2-2a_2}{n(n-1)(n-2)}\right)\beta_n(X)\\
&= \frac{\alpha_{n}(X)}{2} +\frac{1}{n(n-1)} \left(a_2-\binom{a_1}{2}\right)\alpha_n(X) 
+\frac{1}{n(n-1)(n-2)}\left(n(1-a_1)+2\binom{a_1}{2}-2 a_2\right)\beta_n(X).
\end{align*}

\noindent The function $f_n(X)$ is given by 
\[f_n(X)=\frac{1}{n(n-1)} \left(a_2-\binom{a_1}{2}\right)\alpha_n(X) 
+\frac{1}{n(n-1)(n-2)}\left(n(1-a_1)+2\binom{a_1}{2}-2 a_2\right)\beta_n(X).\]

\noindent Now Lemma~\ref{lem:zlambda-formulaS} guarantees that the two sums \[ \sum_{\lambda\vdash n} z_\lambda^{-1} (1-a_1),\ \ \sum_{\lambda\vdash n} z_\lambda^{-1} \left(a_2-\binom{a_1}{2}\right)\] vanish identically. Since  $\alpha_n(X)$ and $\beta_n(X)$ are independent of $\lambda$, we obtain

\[ 
 \sum_{\lambda\vdash n} z_{\lambda}^{-1} f_n(X)=0 \quad \text{and} \quad \sum_{\lambda \vdash n}  z_{\lambda}^{-1}  \E_{\lambda}[X] =\frac{\alpha_{n}(X)}{2}, \]
\noindent as claimed. 
\end{proof}

 Now let $Y$ be any of the cyclic permutation statistics considered in Section \ref{sec:cyclic}. Arguments analogous to the above give us the following.

\begin{proposition}\label{prop:1stmoments-cyclic-stats-Sn} For any of the cyclic statistics $Y$ from Section \ref{sec:cyclic}, the first moment on the conjugacy class $C_\lambda$  for each $\lambda=(1^{a_1},2^{a_2},\ldots)\vdash n$ is of the form 
\[\E_\lambda[Y]= \E_{S_n}[Y] + g_n(Y),\]
where $g_n(Y)$ is some polynomial of degree at most $1$ in $a_1$ and $a_2$ such that 
$\sum_{\lambda\vdash n} z_\lambda^{-1} g_n(Y)=0.$ We have
\begin{enumerate}
    \item $\E_{S_n}[\exc]=\frac{n-1}{2}$, $\E_{S_n}[\aexc]=\frac{n+1}{2}$,
    \item $\E_{S_n}[\cda]=\frac{n-2}{6}=\E_{S_n}[\cdd] $, and
    \item $\E_{S_n}[\cv]=\frac{2n-1}{6}=\E_{S_n}[\cpk] $.
\end{enumerate}
\begin{proof}
    These follow as in Proposition~\ref{prop:Conjclass-to-Sn}, from Theorem~\ref{thm:Elambda-exc} and Theorem~\ref{thm:Elambda-cdasc-cval}, using Lemma~\ref{lem:zlambda-formulaS}.
\end{proof}
\end{proposition}

We conclude this section by noting that we can now also compute the variance of the statistic $\exc$, thanks to  
 the following generating function derived in \cite{CooperJonesZhuang2020}.

 Recall that $C_\lambda$ denotes the conjugacy class in $S_n$ indexed by the partition $\lambda$. 

\begin{proposition}\cite[Corollary 7]{CooperJonesZhuang2020} Let $\lambda $ be a partition of $n$ with $a_1$ parts of size 1. Then 
\[\sum_{w\in C_\lambda} t^{\exc(w)} = \sum_{i=0}^{\lfloor{(n-a_1)/2}\rfloor} \gamma_i t^i (1+t)^{n-a_1-2i},\]
where $\gamma_i= 2^{-n+a_1+2i} |\{w\in C_\lambda: \cv(w)=i\} |$.
\end{proposition}

\noindent From this  we can compute, essentially by differentiating twice to get the generating function for $\exc^2$, the second moment over the conjugacy class $C_\lambda$:
\begin{equation*}
\E_\lambda[\exc^2]=\frac{(n-a_1)(n-a_1+1)}{4} -\frac{1}{2} \E_\lambda[\cv]
=\frac{(n-a_1)^2}{4}+\frac{n-a_1}{4}-\frac{n-a_1+a_2}{6},
\end{equation*}
and therefore the variance
\begin{equation*}
    \operatorname{Var}_\lambda[\exc]=\E_\lambda[\exc^2]-\frac{(n-a_1)^2}{4}=\frac{n-a_1-2a_2}{12}.
\end{equation*}
Hence we obtain, using Lemma~\ref{lem:zlambda-formulaS}, the second moment over all of $S_n$,
\begin{equation*}
    \E_{S_n}[\exc^2]=\frac{(3n-2)(n+1)}{12},
\end{equation*}
and the variance over all of $S_n$,
\begin{equation*}
    \operatorname{Var}_{S_n}[\exc]=\frac{n-2}{12}.
\end{equation*}

\section{Permutation constraints and higher moments}\label{sec:highermoments}

In this section, we examine permutation statistics that track permutations respecting a specified partial function. Somewhat surprisingly, this notion captures the entire class of permutation statistics. This formulation allows us to extend a technique of Fulman \cite[Theorem~3]{FulmanJCTA1998} to establish an independence result for the $k$th moment ($k \geq 1$) across individual conjugacy classes of arbitrary permutation statistics, provided each part of the indexing permutation is sufficiently large. Fulman \cite[Corollary~5]{FulmanJCTA1998} established the analogous result for $d(\omega)$ and $\maj$. In the symmetric case, we also show that these higher moments are polynomials in $n$.

We first start by defining the notion of a  permutation constraint statistic.

\begin{definition} \label{def:AcyclicityConditions} 
Suppose we have a set of pairs $K := \{(i_1,j_1),(i_2,j_2), \dots, (i_\ell,j_\ell)\}$ with each $i_t \in [n], j_t \in [n]$.  We call this a \textit{(permutation) constraint} and say it has \textit{size} $m$ if $K$ contains $m$ pairs. Note that since $K$ is a set, repeated pairs are not allowed.  We say $\omega \in S_n$ \textit{satisfies $K$} if for each $(i_t,j_t)\in K$, $\omega(i_t) = j_t$.  We say that $K$ is \textit{well-defined} if all the $i_t \in [n]$ are distinct and all the $j_t \in [n]$ are distinct; note that some $i_t$ may be equal to some $j_s$. Define the \textit{support} of a constraint $K$ to be the set of all (distinct) $i_t$ and $j_s$.
     
Given a constraint $K$, construct the graph $G(K)$ on vertices $[n]$ by drawing an edge between each pair $(i_t,j_t)$.  We say that $K$ is \emph{acyclic of size $m$} if $K$ is well-defined and $G(K)$ is acyclic with $m$ edges. Note that the graph constructed from a set of acyclic constraints will be a set of disconnected directed paths.
\end{definition}

\begin{ex}
Consider the constraint $K = \{ (1,2), (2,3) \}$ of size $2$. The permutation $(1234)$ satisfies $K$, as $(1234)$ maps $1 \mapsto 2$ (specified by $(1,2) \in K$) and $2 \mapsto 3$ (specified by $(2,3) \in K$). Intuitively, permutations that satisfy $K$ contain $123$ as a subsequence within the \textit{same} cycle. 
\end{ex}

\begin{ex}
Consider the acyclic constraint $K = \{ (1,2), (2,3), (3,4)\}$ of size $3$.  The permutation $(1234)$ satisfies $K$. Now the graph arising from $K$ as in \Def{def:AcyclicityConditions} is acyclic -- in particular, observe that the constraint $(4,1) \not \in K$. Nonetheless, $(1234)$ is a closed cycle. Thus, there may be cycles in the support of a constraint $K$, even when $K$ is itself acyclic. 
\end{ex}

Permutation constraints induce statistics on $S_n$, which we formalize as follows.

\begin{definition}\label{def:PermutationConstraintStatistic}
     Let $\mathcal{C}$ be a set of permutation constraints. The \emph{size} of $\mathcal{C}$, denoted $\text{size}(\mathcal{C})$, is the maximum of the sizes of the constraints in $\mathcal{C}$. Note that while the size of a single constraint $K\in \mathcal{C}$ is simply its size as a set, this is not true for a set of constraints $\mathcal{C}$.
\end{definition}

\begin{definition}
    A \textit{weighted constraint statistic} $X$ is any statistic which can be expressed in the form $\sum_{K \in \mathcal{C}} \wt(K)I_K$ where $\mathcal{C}$ is a set of constraints, $I_K$ is the indicator function that a permutation satisfies the constraint $K$, and weights $\wt(K)\in \mathbb{R}\setminus \{0\}$ for all $K$. In this case, we say $X$ is \emph{realizable} over $\mathcal{C}$. If $X$ can be expressed in this form with $\wt(K)=1$ for all $K\in \mathcal{C}$, then $X$ is the \textit{unweighted constraint statistic} induced by $\mathcal{C}$.
    
    Note that in general, the decomposition $\sum_{K \in \mathcal{C}} \wt(K)I_K$ is not unique. The \emph{size} of a weighted constraint statistic $X$ is defined as
    \[\text{size}(X)=\min\left\{\text{size}(\mathcal{C})\, \bigg| \, X=\sum_{K\in \mathcal{C}}\wt(K)I_K \text{ for }\wt(K)\in \mathbb{R}\setminus \{0\}\right\}.\]
\end{definition}

\begin{remark}
It turns out that the class of weighted constraint statistics actually captures all permutation statistics. Fix $n \geq 1$. For a permutation $\omega \in S_{n}$, consider its graph $\mathcal{G}_{\omega} = \{ (i, \omega(i)) : i \in [n]\}$. The indicator for the constraint induced by $\mathcal{G}_{\omega}$ is precisely the indicator function for the constraint specified by the permutation $\omega$. The class of weighted constraint statistics includes indicator functions for any single permutation, as well as $\mathbb{R}$-linear combinations of them. This in turn captures the algebra of functions from $S_{n} \to \mathbb{R}$.

In this section, we will establish independence results for higher moments of permutation statistics on individual conjugacy classes, provided all parts of the indexing partition are sufficiently large compared to the size of the statistic. Thus, when investigating an individual permutation statistic $X$, it is of interest to exhibit \emph{small} constraint sets that realize $X$. 
\end{remark}

\begin{remark}
Any unweighted constraint statistic $X$ can also be considered as a weighted constraint statistic. In general, the size of $X$ as an unweighted permutation constraint statistic may be different than when viewing it as a weighted constraint statistic, though we only consider the notion of size for weighted constraint statistics. 
\end{remark}

The above definitions are a little abstract and very general, so we first give a few familiar examples.

\begin{ex}
The number of fixed points is a constraint statistic of size $1$.  To see this, let $\mathcal{\mathcal{C}_\text{fix}}$ be the set of all constraints $\{ \{(i,i) \}: i = 1,\dots,n\}$.  Then we have
    $$
        \text{Fix}(\omega) = \sum_{K \in \mathcal{C}_\text{fix}} I_K(\omega).
    $$
\end{ex}

\begin{ex} \label{ex:Cij}
    Let $\mathcal{C}_{i,j}$ be the set of all constraints $\{ \{ (i,a), (j,b) \} $ : $a > b\}$. Then we may express $\des, \maj$, and $\inv$ in terms of these, meaning that these are weighted constraint statistics of size at most $2$ (and indeed, $\des$ and $\inv$ are unweighted). In particular define the following:
    \begin{itemize}
        \item $\mathcal{C}_{\text{inv}} = \cup_{1 \leq i<j \leq n} \mathcal{C}_{i,j}$, and
        \item $\mathcal{C}_{\text{des}} = \cup_{1 \leq i \leq n-1} \mathcal{C}_{i,i+1}$.
    \end{itemize}
    Then setting $\wt(\{(i,a),(j,b)\}):=i$, we obtain
    $$
        \maj(\omega) = \sum_{K \in \mathcal{C}_{\text{des}}} \wt(K)I_K(\omega).
    $$
    Similar formulas exist for $\des$ and $\inv$.  We can also obtain more general statistics such as cyclic descents.  For example,
    $$
        \mathcal{C}_{\cdes} = \mathcal{C}_{\des} \cup \mathcal{C}_{n,1}.
    $$
    Then in a similar manner to before we have that
    $$
        \cdes(\omega) = \sum_{K \in \mathcal{C}_{\cdes}} I_K(\omega).
    $$
    Note that these statistics actually have size equal to $2$.  This fact follows from our work on first moments, combined with Corollary \ref{cor:sizelowerbound} below.
\end{ex}

We give another example of a weighted constraint statistic that is not a weighted inversion statistic: excedance.  

\begin{ex}
    Recall that an excedance is defined as an $i \in [n]$ with $\omega(i) > i$.  We can define the corresponding set of constraints as follows:
    $$
        \mathcal{C}_{\exc} = \cup_{1 \leq i < j \leq n} \{\{(i, j)\}\}. 
    $$
    Then we have that
    $$
        \exc(\omega) = \sum_{K \in \mathcal{C}_{\exc}} I_K(\omega).
    $$
\end{ex}

\begin{remark}
Note that weighted constraint statistics, even those that are realizable over constraints of size $2$, are more general than weighted inversion statistics. Furthermore, permutation statistics realizable over constraints of size $3$ already capture all 14 of the statistics from \cite{blitvic2021permutations}. For instance, we will see below that the number of inversions between excedances where the greater excedance is \emph{linked} (denoted \text{ile}) is realizable  over symmetric constraints of size $3$. 
\end{remark}

\begin{ex}  
The number of inversions between excedances where the greater excedance is \emph{linked} is defined \cite{blitvic2021permutations} by
$$
\text{ile}(\omega) := \# \{(i,j) \in [n] \times [n] : i < j < \omega(j) < \omega(i) \text{ and } \omega^{-1}(j) < j \}.
$$
(Recall from  Section~\ref{sec:Exc} that the linked excedances of \cite{blitvic2021permutations} coincide with the cyclic double ascents of \cite{CooperJonesZhuang2020}.)

We are therefore counting occurrences of $i<j$ with $\omega^{-1}(j) < j < \omega(j) < \omega(i)$.  This means we can define the following set of all constraints:
$$
\mathcal{C}_{\text{ile}} := \cup_{1 \leq i < j \leq n} \{ \{ (i,a), (j,b), (k,j) \} : k < j < b < a \}.
$$
In a similar manner to before this gives
$$
\text{ile}(\omega) = \sum_{K \in \mathcal{C}_{\text{ile}}} I_K(\omega).
$$
\end{ex}

\begin{ex}\label{ex:Denert}\cite{FoataZeilb1990}  The \emph{Denert} statistic is defined by
\begin{equation*}
\begin{split}\mathrm{den}(\omega):= 
&\,\#\{1\le i<j\le n: \omega(j)<\omega(i)\le j\}\\
+&\,\#\{1\le i<j\le n: \omega(i)\le j<\omega(j)\}\\
+&\,\#\{1\le i<j\le n: j< \omega(j)<\omega(i)\}
\end{split}
\end{equation*}
The statistic $\mathrm{den}$ has the property that the joint distributions of the pairs $(\exc, \mathrm{den})$ and $(\des, \maj)$ coincide.  Such pairs are called Euler-Mahonian in the literature 
\cite{FoataZeilb1990}.

Observe that $\mathrm{den}$ may be realized as an unweighted constraint statistic induced by a constraint set of size 2, since we have 
\[\mathrm{den}(\omega)=\sum_{K\in \mathcal{C}_{\mathrm{den}}} I_K(\omega)\]
for the set of constraints
\begin{equation*}
\begin{split} \mathcal{C}_{\mathrm{den}}:=
&\cup_{1\le i<j\le n} \{\{(i,a), (j,b)\}: b<a\le j\}\\
\bigcup & \cup_{1\le i<j\le n} \{\{(i,a), (j,b)\}: a\le j <b\}\\
\bigcup &\cup_{1\le i<j\le n} \{\{(i,a), (j,b)\}: j<b<a\}.
\end{split}
\end{equation*}
    
\end{ex}

We now give a relatively simple observation.

\begin{proposition}
    Let $K$ be a well-defined constraint of size $m$.  Then we have:
    $$
        \pr_{S_n}[\omega \text{ satisfies } K] =  \frac{1}{n(n-1)(n-2)\dots(n-m+1)}.
    $$
\end{proposition}
\begin{proof}
    Let $K :=\{ (i_1,j_1),(i_2,j_2), \dots, (i_m,j_m)\}$, and suppose $\omega$ satisfies $K$.  This means that we have $\omega(i_t) = j_t$ for $t=1,\dots,m$, which is possible since the constraint is well-defined.  The number of permutations which satisfy these $m$ values is just the number of permutations on the remaining $n-m$ symbols, which is $(n-m)!$.  Therefore the probability of a random permutation satisfying $K$ is $(n-m)!/n!$ as required.
\end{proof}

We are interested in the behavior of certain constraint statistics on fixed conjugacy classes.  The key result of this section is the following, which says that for $\lambda$ with all parts ``large,'' the probability of a permutation in $C_\lambda$ satisfying a constraint is only dependent on whether the constraint set is acyclic.

\begin{lemma}\label{lem:size}
    Let $\lambda$ have all parts of size at least $m+1$, and let $K$ be a constraint of size $m$.  If $K$ is acyclic then we have
    \begin{align*}
        \pr_{\lambda}[\omega \text{ satisfies } K] = \frac{1}{(n-1)(n-2)\dots(n-m)}.
    \end{align*}
    If $K$ is not acyclic then we have
    $$
        \pr_{\lambda}[\omega \text{ satisfies } K] = 0.
    $$
\end{lemma}
\begin{proof}
We first note that if $K$ is not acyclic, then in order for $\omega$ to satisfy $K$, $\omega$ must contain a cycle induced by constraints in $K$.  Since $K$ has size $m$, then this cycle is of length at most $m$.  However we assumed $\omega$ is of cycle type $\lambda$ with all cycles of length at least $m+1$, so this is not possible.

    Now suppose $K$ is acyclic.  We fix $n$ and then prove this lemma by induction on $m$.  For $m=1$, we will show that
    $$
        \pr_{\lambda}[\omega(i_1) = j_1] = \frac{1}{n-1}.
    $$
    This follows from the fact that conjugating by $(j_1 \, k)$ for any $k \neq i_1, j_1$ maps from the set of $\omega$ with $\omega(i_1) = j_1$ to those with $\omega(i_1) = k$.  Therefore this probability is the same for each $j_1 \neq i_1$, and is zero for $i_1 = j_1$ since $\lambda$ is fixed point free.  Therefore the probability is $1/(n-1)$ as required.
    
    Assume the statement is true for $m-1$.  Let $A = \{(i_1, j_1), \dots, (i_m, j_m) \}$ be an acyclic constraint of size $m$.  Let $\lambda \vdash n$ have all parts of size at least $m+1$, and label the cycles of any permutation in $C_\lambda$ by $c_1, \dots, c_t$.  By Definition~\ref{def:PermutationConstraintStatistic}, we have
\begin{align*}
\pr_{\lambda}[\omega \text{ satisfies } A] &= \pr_{\lambda}\left[ \bigwedge_{\ell=1}^m \omega(i_\ell) = j_\ell \right] \\
&= \pr_{\lambda}\left[\bigwedge_{\ell=1}^{m-1} \omega(i_\ell) = j_\ell \biggr| \, \omega(i_m) = j_m \right] \cdot \pr_{\lambda}[\omega(i_m) = j_m] \\
&= \frac{1}{n-1} \sum_{h=1}^t \pr_{\lambda}\left[\bigwedge_{\ell=1}^{m-1} \biggr( \omega(i_\ell) = j_\ell \land i_m \in c_h \biggr) \biggr| \, \omega(i_m) = j_m \right] \\
&= \frac{1}{n-1} \sum_{h=1}^t \pr_{\lambda}\left[\bigwedge_{\ell=1}^{m-1} \omega(i_\ell) = j_\ell \, \biggr| \, i_m \in c_h \land \omega(i_m) = j_m \right] 
        \cdot \pr_{\lambda}[i_m \in c_h \mid \omega(i_m) = j_m]. 
\end{align*}
Notice that $A' := \{(i_1,j_1), \dots, (i_{m-1},j_{m-1})\}$ is an acyclic constraint of size $m-1$.  Let $\lambda'(h)$ be the partition obtained by reducing the size of the $h^{th}$ part of $\lambda$ by one.  This is a partition of an $(n-1)$-element set (though perhaps not $[n-1]$) with all parts of size at least $m-1$. 
    It is then fairly straightforward to see that
\[
\pr_{\lambda}\left[\bigwedge_{\ell=1}^{m-1} \omega(i_\ell) = j_\ell \, \biggr| \, i_m \in c_h \land \omega(i_m) = j_m \right] = \pr_{\lambda'(h)}[\omega \text{ satisfies } A'] = \frac{1}{(n-2)(n-3)\dots(n-m)}, 
\]
    where the last equality follows by the induction hypothesis. Note that the first term is $n-2$, as the probability is in $S_{n-1}$. Putting this altogether gives
    \begin{align*}
        \pr_{\lambda}[\omega \text{ satisfies } A] &=  \frac{1}{n-1} \sum_{h=1}^t \frac{1}{(n-2)(n-3)\dots(n-m)} \frac{\lambda_h}{n} \\
        &= \frac{1}{(n-1)(n-2)\dots(n-m)}. 
    \end{align*}
    This completes the inductive step and the proof.
\end{proof}

As a consequence, we obtain that for each $k$, the $k$th moment of these statistics is independent of conjugacy class, as long as the cycles are sufficiently long.

\begin{theorem} \label{thm:IndependenceHigherMoments}
    Let $X$ be a permutation statistic that is realizable over a constraint set of size $m$, and fix $k \geq 1$. If $\lambda\vdash n$ has all parts of size at least $mk+1$, then $\E_{\lambda}[X^k]$ is independent of $\lambda$.
\end{theorem}

\begin{proof}
Express $X=\sum_{P\in \mathcal{C}} \wt(P) I_P$, where $\text{size}(\mathcal{C})=m$. We start by decomposing the variable $X^k$ into random indicator variables.
\begin{align*}
\E_{\lambda}[X^k] &= \sum_{P_1 \in \mathcal{C}} \sum_{P_2 \in \mathcal{C}} \dots \sum_{P_k \in \mathcal{C}} \prod_{i=1}^k \wt(P_i) \cdot \E_{\lambda}[I_{P_i}] \\
&= \sum_{P_1 \in \mathcal{C}} \sum_{P_2 \in \mathcal{C}} \dots \sum_{P_k \in \mathcal{C}} \left(\prod_{i=1}^k \wt(P_i)\right) \pr_{\lambda}\left[\bigwedge_{i=1}^k \omega \text{ satisfies } P_i \right].
\end{align*}
We therefore continue by evaluating each of the individual probabilities in the sum.  

Fix some tuple $P_1,P_2, \dots, P_k$, and let $Y$ be the union of all of these constraints excluding repeats.  Write $Y = \{ (i_1,j_1), \dots, (i_s, j_s) \}$, noting that all the pairs are distinct.  We split into three cases.
\begin{itemize}    
\item \textbf{Case 1}: Suppose first that $Y$ is not well defined. Then there must be some repeated $i_t$ or $j_t$. Since we excluded repeats, there must be pairs of the form $\{(i_t,a),(i_t,b)\}$ or $\{(a,j_t),(b,j_t)\}$. However $\omega(i_t)$ and $\omega^{-1}(j_t)$ can only take one value, so the probability of $Y$ being satisfied is zero.
    
\item \textbf{Case 2}: Suppose instead that $Y$ is not acyclic. Then by \Lem{lem:size}, we have that $\pr_{\lambda}[w \text{ satisfies } Y] = 0$.

\item \textbf{Case 3}: $Y$ is well defined, and no subsets of the values in $\omega(i_1) = j_1 , \omega(i_2) = j_2 , \dots, \omega(i_s) = j_s$ form a cycle.  Then this is a set of acyclic constraints of size at most $mk$.  By the previous proposition we therefore have that
    $$        \pr_{\lambda}[ \omega \text{ satisfies } Y] = \frac{1}{(n-1)(n-2)\dots(n-s)}.
    $$
\end{itemize}    

In particular, none of these probabilities depend on the choice of $\lambda$, so the result follows.  
\end{proof}

\begin{corollary}\label{cor:sizelowerbound}
    Let $X$ be a permutation statistic, and let $\lambda \vdash n$. Suppose that $\E_{\lambda}[X]$ depends on the number of parts of $\lambda$ of size $m$. Then any constraint set realizing $X$ must have size at least $m$.
\end{corollary}

\begin{remark}
Let $\mathcal{C}$ be a constraint set of size $m$. Clearly, if we can express $X=\sum_{P\in \mathcal{C}}I_P$, then any minimum-sized constraint set realizing $X$ has size at most $m$.

The above corollary shows that calculating the first moment of $X$ even just on specific conjugacy classes allows us to obtain a lower bound on the size of $X$. This approach allows us to explicitly calculate the size for many statistics. 

Once we have determined the size of $X$, we can then apply Theorem \ref{thm:IndependenceHigherMoments}, so we see that information on the higher moments of $X$ can be obtained from the first moment, further highlighting the importance of the latter.   
\end{remark}

\begin{remark}\label{remark:unweighted}
It will be useful later to write the expectation $\mathbb{E}_{\lambda}[X^{k}]$ from \Thm{thm:IndependenceHigherMoments}  more explicitly in the unweighted case, so we do this.

Let $\mathcal{A}$ be the set of all the acyclic constraints from amongst the tuples $P_1, \dots, P_k$ in the sum.  Let $\mathcal{A}_t$ be the set of all the acyclic constraints in $\mathcal{A}$ of size $t$. Using the three previous cases, we may write the required expectation as
    \begin{align*}
        \E_{\lambda}[X^k] &= \sum_{P \in \mathcal{A}} \pr_{\lambda}[\omega \text{ satisfies } P] \\
        &= \sum_{t} \frac{|\mathcal{A}_t|}{(n-1)(n-2)\dots(n-t)}.
    \end{align*}

\noindent This number is independent of the choice of $\lambda$ as long as it has parts of size at least $mk+1$. Observe that taking $X(\omega) = \maj(\omega)$ or $X(\omega) = d(\omega)=1+\des(\omega)$ yields \cite[Theorem~2]{FulmanJCTA1998}.
\end{remark}

We continue by showing that when a statistic is \emph{symmetric}, these moments are polynomial in $n$. We now define this precisely.

\begin{definition}
     Let $a_1,\ldots,a_{n_0}\in [n]$. A function $f:\{a_1,\ldots,a_{n_0}\}\rightarrow [n]$ is \emph{order-preserving} when $a_i<a_j$ if and only if $f(a_i)<f(a_j)$ for all $i,j\in [n_0]$. Note that any such function must be injective.
\end{definition}

\begin{definition} \label{def:SymmetricConstraintStatistic}
     Let $\mathcal{C}$ be a set of permutation constraints, and let $X$ be the unweighted constraint statistic induced by $\mathcal{C}$. Take some $P = \{ (i_1,j_1) \dots (i_\ell, j_\ell)\} \in \mathcal{C}$.  Let the distinct symbols amongst the $i_1, \dots, i_\ell, j_1, \dots, j_\ell$ be $1 \leq a_1 < a_2 < \dots < a_{n_0} \leq n$. If $f(P) := \{ (f(i_1),f(j_1)) \dots (f(i_\ell), f(j_\ell)) \} \in \mathcal{C}$ for all such choices of $P \in \mathcal{C}$ and order-preserving $f:\{a_1,\ldots,a_{n_0}\}\rightarrow [n]$, then we say that $X$ is \emph{symmetric}. 
\end{definition}

 \noindent We start by examining how this definition relates to some familiar statistics.

\begin{itemize}
    \item Inversions are symmetric: take any $P =\{ (a,b),(c,d) \} \in \mathcal{C}_{\text{inv}}$ and any order preserving injection $f: \{a,b,c,d\} \rightarrow [n]$.  Then we must have $a < c, b > d$, so $f(a) < f(c), f(b) > f(d)$.  Therefore $f(P)=\{ (f(a),f(b)),(f(c),f(d))\} \in \mathcal{C}_{\text{inv}}$.
    
    \item Descents cannot be realized as symmetric constraint statistics using constraints of size $2$. Let $\mathcal{C}_{i,j}$ be as defined in Example~\ref{ex:Cij}.  For example, take $P = \{(1,5), (2,4)\} \,{\in\mathcal{C}_{1,2}}\subseteq\mathcal{C}_{\text{des}}$.  Let $1 < 3 < 4 < 5$, with $f(1)=1, f(2)=3, f(4)=4, f(5)=5$.  Then $f(P) = \{ (1,5),(3,4) \}\, {\in\mathcal{C}_{1,3}} \not\subseteq\mathcal{C}_{\text{des}}$. We may iterate on this argument, replacing $(1,5),(2,4)$ with arbitrary values respecting the same relative ordering. It is not clear  whether descents can be realized using a symmetric constraint set of larger size.

    \item The number of inversions between excedances, as defined in \cite{blitvic2021permutations}, is symmetric.  This is because the constraints for this statistic are exactly the $(a,b),(c,d)$ with $a < c < d < b$, so the images of these elements under an order-preserving $f$ will give another valid constraint.
\end{itemize}

\noindent Given a symmetric permutation constraint statistic on $S_{n_0}$, there is also a natural way of extending this statistic to any $S_n$.

\begin{definition} \label{def:SymmetricExtension}
     Let $X$ be a symmetric permutation constraint statistic on $S_{n_0}$ induced by some $\mathcal{C}$ supported on $[n_0]$.  Then for any $S_n$, we can define a symmetric permutation constraint statistic $X_n$ on $S_n$ by starting with the set of constraints $\mathcal{C}$ for $X$ and constructing the following set of constraints $\mathcal{C}_n$ for $X_n$.
     \begin{itemize}
         \item If $n\leq n_0$, then let $\mathcal{C}_n$ contain all $P\in \mathcal{C}$ with support contained in $[n]$.
         \item If $n>n_0$, then let $\mathcal{C}_n$ contain all $P \in \mathcal{C}$, as well as all $f(P)$ for all order-preserving functions $f:[n_0] \rightarrow [n]$.  Note that we exclude repeated constraints in $\mathcal{C}_n$.
     \end{itemize} 
     Then by construction each $X_n$ is symmetric.
     We call $(X_n)$ a \emph{symmetric extension of $X$}.
\end{definition}

\begin{ex}
While the previous definition seems technical, there are several natural examples.
\begin{itemize}
    \item Consider the constraint $K=\{(1,2)\}$, and define the statistic $X$ on $S_2$ by $X=I_K$. Then the $(X_n)$ are the excedance statistics.
    \item Fix $\omega\in S_{m}$. Let $\mathcal{C}$ be the constraints of size $m$ in $S_{2m}$ that induce the permutation pattern statistic for $\omega$ in $S_{2m}$. Then each statistic in $(X_n)$ is the number of appearances of the permutation pattern $\omega$ for a given element in $S_n$. Note that choosing $\omega=(12)\in S_2$ results in the usual inversion statistics on $S_n$. 
\end{itemize}
\end{ex}

\begin{remark}
The preceding examples show that symmetric permutation constraint statistics are more general than permutation pattern statistics, as excedances cannot be expressed as a permutation pattern. See Remark~\ref{rmk:PermutationPatterns} for more discussion, as well as a comparison of our work with that of Gaetz and Pierson \cite{GaetzPierson}.
\end{remark}

\begin{remark}
In general, it is necessary to consider symmetric extensions starting from some sufficiently large $n_0$. Observe that both $\{(1,2)\}$ and $\{(1,2), (2,1)\}$ induce $\inv$ on $S_{2}$. However, the symmetric extension of $\{(1,2)\}$ yields the excedance statistic, while the symmetric extension of $\{ (1,2), (2,1)\}$ realizes transpositions. In the preceding example, we see that the symmetric extension starting with the inversion statistic on $S_4$ results in the inversion statistics on all $S_n$.
\end{remark}

With this definition in hand, we now show that when all parts of a partition are sufficiently large, the moments of any statistic constructed in this manner are given by a single polynomial dependent only on $n$.

\begin{theorem} \label{thm:PropHigherMomentPolynomial}
Fix $k, m \geq 1$. Let $(\lambda_{n})$ be a sequence of partitions, where $\lambda_{n} \vdash n$ and all parts of $\lambda_{n}$ have size at least $mk+1$. Let $(X_{n})$ be a symmetric extension of a symmetric permutation statistic $X=X_{n_0}$ induced by a constraint set of size $m$. There exists a polynomial $p_X(n)$ depending only on $X$ such that $p_X(n) = \E_{\lambda_{n}}[X_{n}^{k}]$.
\end{theorem}

\begin{proof}
    As in Theorem \ref{thm:IndependenceHigherMoments}, it suffices to consider $\mathcal{A}_n= \bigcup_{i} P_{n,i}$, where the union runs over all well-defined acyclic $k-$tuples of constraints in $X_n$. Let $\mathcal{A}_{n,t}\subseteq \mathcal{A}_n$ be the constraints of size $t$. Note that each constraint $P \in \mathcal{A}_n$ is a tuple of constraint, and multiple constraints may involve the same elements. Recall that the support of a constraint $P = \{ (i_1,j_1), \dots, (i_\ell,j_\ell) \} \in \mathcal{A}_n$ is the set of distinct elements among the $i_1, \dots, i_\ell, j_1, \dots, j_\ell$.  Define $\mathcal{A}_{n,t,s}\subseteq \mathcal{A}_{n,t}$ be the constraints of size $t$ with support on $s$ elements, where acyclicity of elements in $\mathcal{A}_{n,t}$ implies $t+1\leq s\leq 2t$.
    Then we have from Remark \ref{remark:unweighted} that
    \begin{equation}\label{nts-equation}
    \begin{split}
        \E_{\lambda_n}[X_n^k] &= \sum_{t=1}^{mk} \frac{|\mathcal{A}_{n,t}|}{(n-1)(n-2)\dots(n-t)} \\
         &= \sum_{t=1}^{mk}\left( \frac{1}{(n-1)(n-2)\dots(n-t)} \sum_{s=t+1}^{2t} |\mathcal{A}_{n,t,s}| \right).
        \end{split}
    \end{equation}
    
    Now let $\mathcal{A}_{n,t,s}'\subseteq \mathcal{A}_{n,t,s}$ be the constraints that are supported on $[s]$. Observe that when $n<s$, $\mathcal{A}_{n,t,s}'=\emptyset$, and since $X_n$ is formed as the symmetric extension of $X_{n_0}$, this $\mathcal{A}_{n,t,s}'$ is independent of $n$ for $n\geq s$, so we call this common set $\mathcal{A}_{t,s}$. Furthermore, since $X_n$ is symmetric, for $n\geq s$, we can express
    \[\mathcal{A}_{n,t,s}= \bigcup_{f} \bigcup_{P\in \mathcal{A}_{t,s}} f(P),\]
    where the first union is over all order-preserving $f$. Now as each $P$ uses all elements of $[s]$ and each $f$ is determined by its image in $[n]$, we have that $f_1(P_1) = f_2(P_2)$ can only occur if $f_1 = f_2$ and $P_1 = P_2$. Then letting $a_{t,s}=|\mathcal{A}_{t,s}|$, we have that for $n\geq s$,
    \[|\mathcal{A}_{n,t,s}|={n\choose s} a_{t,s},\]
    as there are ${n \choose s}$ order-preserving functions $f:[s]\to [n]$. Letting $I_s(n)$ be the indicator function for $n\geq s$, we see that (\ref{nts-equation}) can be rewritten as
    \begin{equation}\label{ts-equation}
    \begin{split}
        \E_{\lambda_n}[X_n^k] &= \sum_{t=1}^{mk}\left( \frac{1}{(n-1)(n-2)\dots(n-t)} \sum_{s=t+1}^{2t} {n\choose s}a_{t,s}I_s(n)\right) \\
        & = \sum_{t=1}^{mk} \sum_{s=t+1}^{2t} {n\choose s}\frac{a_{t,s}I_s(n)}{(n-1)(n-2)\dots(n-t)}.
    \end{split}
    \end{equation}
    Observe that $s\geq t+1$, and when $n\geq s$, we have that
    \begin{equation*}
        \begin{split}
            \binom{n}{s}\cdot  \frac{1}{(n-1)(n-2)\dots(n-t)} & = \frac{1}{s!} \cdot  \frac{n(n-1)(n-2)\dots(n-s+1)}{(n-1)(n-2)\dots(n-t)} \\
            & = \frac{1}{s!}\cdot n(n-t-1)\ldots (n-s+1)
        \end{split}
    \end{equation*}
    is a polynomial in $n$ of degree $s-t$. Furthermore, $\lambda_n$ has all parts of size at least $mk+1>t$, so $n\geq mk+1 >t$. When values of $n$ with $t<n<s$ are substituted, the above polynomial vanishes. Hence, we can rewrite (\ref{ts-equation}) and omit the $I_s$ indicator function to obtain
    \begin{equation}\label{ts-equation2}
    \begin{split}
        \E_{\lambda_n}[X_n^k] & = \sum_{t=1}^{mk} \sum_{s=t+1}^{2t}\frac{a_{t,s}}{s!}\cdot n(n-t-1)\ldots (n-s+1).
    \end{split}
    \end{equation}
    We conclude that (\ref{ts-equation}) is a polynomial in $n$ of degree \[\max_{P \in \mathcal{A}_{2mk}} (|\supp(P)| - \text{size}(P))\leq mk.\qedhere \]
\end{proof}
\begin{remark}
The proof of the preceding result gives a method for finding $p_X(n)$, which we illustrate with an example. Consider the mean of the inversion statistic on conjugacy classes $\lambda_n$ with cycle lengths of at least $3$, so that $m=2$ and $k=1$ in (\ref{ts-equation2}). In the summation of (\ref{ts-equation}), the only nonzero values involve $t=2$, which implies $s\in \{3,4\}$. Of the constraints in $\inv$ using only values in the sets $[3]$ and $[4]$, we see that the acyclic ones that use all values are
\[\mathcal{A}_{2,3}=\left\{ \{(1,3),(2,1)\},\{(1,2),(3,1)\},\{(1,3),(3,2)\},\{(2,3),(3,1)\}\right\},\]
\[\mathcal{A}_{2,4}=\left\{ \{(1,4),(2,3)\},\{(1,4),(3,2)\},\{(1,3),(4,2)\},\{(2,4),(3,1)\},\{(2,3),(4,1)\},\{(3,2),(4,1)\}\right\}.\]
Then (\ref{ts-equation2}) becomes
\[\E_{\lambda_n}[\inv]=\frac{4}{3!}\cdot n+\frac{6}{4!}\cdot n(n-3)=\frac{3n^2-n}{12},\]
which agrees with our Corollary \ref{cor:E-des-maj-inv}. For higher moments, explicit description of acyclic constraints in terms of $k$-tuples becomes significantly more complex, and this method becomes computationally very difficult.
\end{remark}

In the case of certain statistics such as inversions, we can determine much more about the structure of this polynomial. 

\begin{proposition} \label{prop:HigherMomentsSectionInversions}
    Let $\lambda$ be a partition of $n$ with all parts of size at least $2k+1$. Then $\E_{\lambda}[\mathrm{inv}^k]$ is a polynomial in $n$ of degree $2k$ with leading coefficient $2^{-2k}$.
\end{proposition}

\begin{proof}
    The polynomiality follows from \Thm{thm:PropHigherMomentPolynomial}. From the proof of this Theorem we also have that
    \begin{align}\label{inv-moments}
        \E_{\lambda}[\text{inv}^k] = \sum_{t=1}^{2k} \sum_{s=t+1}^{2t} \frac{a_{t,s}}{s!}\cdot n(n-t-1)\ldots (n-s+1).
    \end{align}
    Recall that $a_{t,s}$ is the number $k$-tuples $\{(a,b),(c,d)\}$ with $a < c, b > d$ that consist of $t$ distinct pairs and use exactly the elements in $[s]$. The degree of this polynomial corresponds to when $s-t \leq 2k$ is maximal. Note that $s-t=2k$ can occur only when $s=4k$ and $t=2k$, so it suffices to show that $a_{2k,4k}$ is nonzero. Hence, we consider $2k$ distinct pairs using all elements in $[4k]$.

    There are $\binom{4k}{4,4,\dots,4}$ ways to partition $[4k]$ into $k$ sets of four symbols.  For each set of four symbols $\{a,b,c,d\}$ suppose that $a<b<c<d$. Then there will be $6$ ways to put this set into two pairs which relate to an inversion constraint, which are
    $$ \{(a,c) , (d,b)\},\, \{(a,d) , (b,c)\}, \,  \{(a,d) , (c,b)\},\, \{(b,c) , (d,a)\}, \,  \{(c,b) , (d,a)\},\, \{(b,d) , (c,a)\}.
    $$
    Therefore in total we have $a_{2k,4k} = \binom{4k}{4,4,\dots,4} 6^k = (4k)!/4^k$.  Substituting this back into (\ref{inv-moments}) gives a leading coefficient of $1/4^k$ for the $x^{2k}$ term as required.
\end{proof}

As an application, we can use polynomial interpolation on $2k+1$ values of $n$ to explicitly compute $\E_{\lambda}[\inv^{k}]$ when all parts of $\lambda$ have size at least $2k+1$. The case of the second moment of $\inv$ is given below.

\begin{corollary} \label{cor:InversionSecondMoment}
    Let $\lambda$ be a partition of $n$ with all parts of size at least $5$. Then
    \[\E_{\lambda}[\inv^2]=\frac{1}{16}n^{4}-\frac{1}{72}n^{3}-\frac{1}{80}n^{2}-\frac{49}{360}n,\]
    and consequently, 
    \[\operatorname{Var}_{\lambda}[\inv]=\frac{1}{36}n^3-\frac{7}{360}n^2-\frac{49}{360}n.\]
\end{corollary}

\begin{proof}
We consider the conjugacy class $C_{(n)}$ corresponding to full cycles in $S_n$. Using code, we find the following values:
\begin{align*}
&\E_{(5)}[\inv^2]=109/3, \\ 
&\E_{(6)}[\inv^2]=1151/15, \\ 
&\E_{(7)}[\inv^2]=2156/15, \\
&\E_{(8)}[\inv^2]=247, \\ 
&\E_{(9)}[\inv^2]=3977/10, 
\end{align*} 
\noindent The result for $\E_{\lambda}[\inv^2]$ follows by polynomial interpolation, and $\operatorname{Var}_{\lambda}[\inv]=\E_{\lambda}[\inv^2]-(\E_{\lambda}[\inv])^2$ then follows by direct calculation.
\end{proof}

\begin{remark}
We compare \Cor{cor:InversionSecondMoment}  with Feller's corresponding result for the full $S_n$ \cite[p.~257, equations (6.1)-(6.3)]{Feller}:
\[\E_{S_{n}}[\inv]=\frac{1}{4}n(n-1),\]

\[\E_{S_{n}}[\inv^2]=\frac{1}{16}n^{4}-\frac{7}{72}n^{3}+\frac{5}{48}n^{2}-\frac{5}{72}n,\]

\[\operatorname{Var}_{S_{n}}[\inv]=\frac{1}{72}(2n^3+3n^2-5n).\]
We note that leading terms coincide.
\end{remark}

\section{Conclusion}

In this paper, we investigated the distributions of various permutation statistics on individual conjugacy classes. We first introduced general notions of permutation statistics, including (i) weighted inversion statistics, which generalized inversions, major index, descents, and baj, and (ii) permutation constraints. We utilized  the notion of permutation constraints to reason about arbitrary permutation statistics. Precisely, we showed that the higher moments are  independent of the conjugacy class indexed by the partition $\lambda \vdash n$, provided all parts of $\lambda$ are sufficiently large. For permutation statistics realizable over symmetric constraints, we were further able to establish polynomiality for the higher moments on individual conjugacy classes indexed by $\lambda \vdash n$, again provided that all parts of $\lambda$ are sufficiently large. Our work leaves open several questions.

In \Prop{prop:Conjclass-to-Sn}, we showed that for any conjugacy class $\lambda$ and a weighted inversion statistic $X$, $\E_{\lambda}[X]$ can be written as $\E_{S_{n}}[X]$ plus some error term $f_{n}^{X}(a_{1}, a_{2})$, which is a degree $2$ polynomial depending only on $X$ and $a_{i}$ ($i = 1, 2$), the number of cycles of size $i$ in $\lambda$. As our independence results in Section~\ref{sec:highermoments} require that all parts of $\lambda$ be sufficiently large, we suspect that \Prop{prop:Conjclass-to-Sn} can be extended in the following manner.

\begin{problem}
Show that $\E_{\lambda}[X^{k}] = \E_{S_{n}}[X^{k}] + f_{n}^{X^{k}}(a_{1}, \ldots, a_{2k})$, where $a_{i}$ is the number of cycles of length $i$ in $\lambda$, and $f_{n}^{X^{k}}$ is a polynomial of degree at most $2k$,
(necessarily) satisfying the condition \[\sum_{\lambda\vdash n} z_{\lambda}^{-1} f_{n}^{X^{k}}(a_{1}, \ldots, a_{2k})=0.\]
\end{problem}

Our technique in establishing \Prop{prop:Conjclass-to-Sn} required detailed case analysis. Moving to even the second moment, the number of cases grows substantially. It would be of interest to find a tractable technique that easily extends to higher moments.

As we have not only an independence result, but also polynomiality on the higher moments of permutation statistics realizable over symmetric constraint sets, it seems plausible that such statistics admit a nice asymptotic distribution. In particular, a central limit theorem for descents on individual conjugacy classes is known \cite{FulmanJCTA1998, Kim2019DistributionOD, KimLee2020}. We thus ask the following.

\begin{problem}
Fix $k, m \geq 1$. Let $(X_{n})$ be a symmetric extension of a symmetric permutation statistic of size $m$. Let $\lambda_{n}$ be a partition of $n$, with each part of size at least $mk+1$. Establish a central limit theorem for $(X_{n})$ on $\lambda_{n}$.
\end{problem}

While we have established that a number of statistics such as $\inv$, $\exc$, $\aexc$, $\text{cdasc},$ and $\text{cddes}$ are symmetric, we have been unable to show that any of the statistics in this paper are \textit{not} symmetric. In particular, we do not have tractable conditions to show that a permutation statistic is not symmetric. Thus, we ask the following.

\begin{problem}
Provide a characterization of when a permutation statistic is realizable over a symmetric constraint set.
\end{problem}

In light of \Thm{thm:PropHigherMomentPolynomial} and the fact that the first moment of $\cdes$ is a rational function on any individual conjugacy class (\Thm{thm:cDes-1st-momentS}), we have that the family $(\cdes_{n})$ cannot be realized as the symmetric extension of any permutation statistic $X$. We conjecture that no individual $\cdes_{m}$ is itself symmetric. However, it is not clear how to establish this. Furthermore, we conjecture that $\des, \maj, \baj,$ and $\baj-\inv$ are not realizable over any symmetric permutation constraints or as the symmetric extensions of any permutation statistic.

Since our work in this paper establishes results for the  Coxeter group of type $A$, it is natural to ask the following.

\begin{problem}
Extend the results of this paper to other Coxeter groups.
\end{problem}

It is likely that the calculations would need to be updated to the setting of the given family of Coxeter groups being considered, but that the techniques in this paper might still apply. Ideally, one might hope for a general technique that can handle all  Coxeter groups without redoing the calculations for each such family.

Given a statistic $X$ on the symmetric group $S_n$, the first moments $\E_\lambda[X]$ are class functions, and may thus be interpreted as the character of a possibly virtual representation of $S_n$. Equation~\eqref{eqn:lambda-to-Sn}, which gives the first moment of $X$ on all of $S_n$,  is then precisely the multiplicity of the trivial module in some (virtual) representation of $S_n$. Thus one could ask if there is a representation-theoretic interpretation of our results, beyond the connection with character polynomials as in \cite{GaetzPierson}.

\begin{problem}
Investigate representation-theoretic interpretations of these results.
\end{problem}

\bibliographystyle{alphaurl}
\bibliography{Bibliography}

\begin{thebibliography}{DMP95}

\bibitem[BD92]{BayerDiaconis}
Dave Bayer and Persi Diaconis.
\newblock Trailing the dovetail shuffle to its lair.
\newblock {\em The Annals of Applied Probability}, 2, 05 1992.
\newblock \href {https://doi.org/10.1214/aoap/1177005705}
  {\path{doi:10.1214/aoap/1177005705}}.

\bibitem[BKS20]{BKS2020}
Sara~C. Billey, Matja\v{z} Konvalinka, and Joshua~P. Swanson.
\newblock Asymptotic normality of the major index on standard tableaux.
\newblock {\em Adv. in Appl. Math.}, 113:101972, 36, 2020.
\newblock \href {https://doi.org/10.1016/j.aam.2019.101972}
  {\path{doi:10.1016/j.aam.2019.101972}}.

\bibitem[Bre93]{Brenti1993}
Francesco Brenti.
\newblock Permutation enumeration, symmetric functions, and unimodality.
\newblock {\em Pacific J. Math.}, 157(1):1--28, 1993.
\newblock URL: \url{http://projecteuclid.org/euclid.pjm/1102634861}.

\bibitem[BS21]{blitvic2021permutations}
Natasha Blitvi{\'c} and Einar Steingr{\'\i}msson.
\newblock Permutations, moments, measures.
\newblock {\em Transactions of the American Mathematical Society},
  374(08):5473--5508, 2021.

\bibitem[Cel98]{Cellini1998}
Paola Cellini.
\newblock Cyclic {E}ulerian elements.
\newblock {\em European J. Combin.}, 19(5):545--552, 1998.
\newblock \href {https://doi.org/10.1006/eujc.1998.0218}
  {\path{doi:10.1006/eujc.1998.0218}}.

\bibitem[CJZ20]{CooperJonesZhuang2020}
M.~Crossan Cooper, William~S. Jones, and Yan Zhuang.
\newblock On the joint distribution of cyclic valleys and excedances over
  conjugacy classes of {${S}_n$}.
\newblock {\em Adv. in Appl. Math.}, 115:101999, 15, 2020.
\newblock \href {https://doi.org/10.1016/j.aam.2020.101999}
  {\path{doi:10.1016/j.aam.2020.101999}}.

\bibitem[DF91]{DummitFoote}
David~S. Dummit and Richard~M. Foote.
\newblock {\em Abstract {A}lgebra}.
\newblock Prentice Hall, Inc., Englewood Cliffs, NJ, 1991.

\bibitem[DG19]{DiaconisGraham}
Persi Diaconis and Ron Graham.
\newblock {\em 12. The Magic of Charles Sanders Peirce}, pages 161--203.
\newblock Princeton University Press, Princeton, 2019.
\newblock \href {https://doi.org/10.1515/9780691194417-014}
  {\path{doi:10.1515/9780691194417-014}}.

\bibitem[DMP95]{DiaconisPersiGrath}
Persi Diaconis, Michael McGrath, and Jim Pitman.
\newblock Riffle shuffles, cycles, and descents.
\newblock {\em Combinatorica}, 15(1):11–29, mar 1995.
\newblock \href {https://doi.org/10.1007/BF01294457}
  {\path{doi:10.1007/BF01294457}}.

\bibitem[DP86]{diaconis1986permutations}
Persi Diaconis and JW~Pitman.
\newblock Permutations, record values and random measures.
\newblock {\em Unpublished lecture notes, Statistics Department, University of
  California, Berkeley}, 1986.

\bibitem[Fel68]{Feller}
William Feller.
\newblock {\em An Introduction to Probability Theory and Its Applications},
  volume~1.
\newblock Wiley, January 1968.
\newblock URL:
  \url{http://www.amazon.ca/exec/obidos/redirect?tag=citeulike04-20{\&}path=ASIN/0471257087}.

\bibitem[Foa68]{FoataPAMS1968}
Dominique Foata.
\newblock On the {N}etto inversion number of a sequence.
\newblock {\em Proc. Amer. Math. Soc.}, 19:236--240, 1968.
\newblock \href {https://doi.org/10.2307/2036179} {\path{doi:10.2307/2036179}}.

\bibitem[Foa77]{FoataMahonian}
Dominique Foata.
\newblock Distributions eul\'eriennes et mahoniennes sur le groupe des
  permutations.
\newblock In Martin Aigner, editor, {\em Higher Combinatorics}, pages 27--49,
  Dordrecht, 1977. Springer Netherlands.
\newblock \href {https://doi.org/10.1007/978-94-010-1220-1_2}
  {\path{doi:10.1007/978-94-010-1220-1_2}}.

\bibitem[FS70]{Foata1970TheorieGD}
Dominique Foata and Marcel~Paul Sch{\"u}tzenberger.
\newblock {\em Th\'eorie {G}\'eom\'etrique des {P}olyn\^omes {E}ul\'eriens}.
\newblock Springer Berlin, Heidelberg, 1970.
\newblock \href {https://doi.org/10.1007/BFb0060799}
  {\path{doi:10.1007/BFb0060799}}.

\bibitem[Ful98]{FulmanJCTA1998}
Jason Fulman.
\newblock The distribution of descents in fixed conjugacy classes of the
  symmetric groups.
\newblock {\em J. Combin. Theory Ser. A}, 84(2):171--180, 1998.
\newblock \href {https://doi.org/10.1006/jcta.1998.2893}
  {\path{doi:10.1006/jcta.1998.2893}}.

\bibitem[FZ90]{FoataZeilb1990}
Dominique Foata and Doron Zeilberger.
\newblock Denert's permutation statistic is indeed {E}uler-{M}ahonian.
\newblock {\em Stud. Appl. Math.}, 83(1):31--59, 1990.
\newblock \href {https://doi.org/10.1002/sapm199083131}
  {\path{doi:10.1002/sapm199083131}}.

\bibitem[GP23]{GaetzPierson}
Christian Gaetz and Laura Pierson.
\newblock Positivity of permutation pattern character polynomials.
\newblock {\em Advances in Applied Mathematics}, 147:102507, 2023.
\newblock \href {https://doi.org/10.1016/j.aam.2023.102507}
  {\path{doi:10.1016/j.aam.2023.102507}}.

\bibitem[GR93]{GesselReutenauer}
Ira~M. Gessel and Christophe Reutenauer.
\newblock Counting permutations with given cycle structure and descent set.
\newblock {\em Journal of Combinatorial Theory, Series A}, 64(2):189--215,
  1993.
\newblock \href {https://doi.org/10.1016/0097-3165(93)90095-P}
  {\path{doi:10.1016/0097-3165(93)90095-P}}.

\bibitem[GR20]{GaetzRyba}
Christian Gaetz and Christopher Ryba.
\newblock Stable characters from permutation patterns.
\newblock {\em Selecta Mathematica}, 27:1--13, 2020.
\newblock \href {https://doi.org/10.1007/s00029-021-00692-9}
  {\path{doi:10.1007/s00029-021-00692-9}}.

\bibitem[HR22]{hamaker2022characters}
Zachary Hamaker and Brendon Rhoades.
\newblock Characters of local and regular permutation statistics, 2022.
\newblock \href {http://arxiv.org/abs/2206.06567} {\path{arXiv:2206.06567}},
  \href {https://doi.org/10.48550/arXiv.2206.06567}
  {\path{doi:10.48550/arXiv.2206.06567}}.

\bibitem[Kim19]{Kim2019DistributionOD}
Gene~B. Kim.
\newblock Distribution of descents in matchings.
\newblock {\em Annals of Combinatorics}, 23:73--87, 2019.
\newblock \href {https://doi.org/10.1007/s00026-019-00414-1}
  {\path{doi:10.1007/s00026-019-00414-1}}.

\bibitem[KL20]{KimLee2020}
Gene~B. Kim and Sangchul Lee.
\newblock Central limit theorem for descents in conjugacy classes of ${S}_n$.
\newblock {\em Journal of Combinatorial Theory, Series A}, 169:105123, 2020.
\newblock \href {https://doi.org/10.1016/j.jcta.2019.105123}
  {\path{doi:10.1016/j.jcta.2019.105123}}.

\bibitem[Knu98]{KnuthTAOCP3}
Donald~E. Knuth.
\newblock {\em The Art of Computer Programming, Volume 3: (2nd Ed.) Sorting and
  Searching}.
\newblock Addison Wesley Longman Publishing Co., Inc., USA, 1998.

\bibitem[Mac15]{MacMahon}
P.~MacMahon.
\newblock {\em Combinatory analysis}.
\newblock Cambridge University Press, 1915.

\bibitem[Mac16]{MacMahon1916}
P.~A. MacMahon.
\newblock Two {A}pplications of {G}eneral {T}heorems in {C}ombinatory
  {A}nalysis: (1) {T}o the {T}heory of {I}nversions of {P}ermutations; (2) {T}o
  the {A}scertainment of the {N}umbers of {T}erms in the {D}evelopment of a
  {D}eterminant which has {A}mongst its {E}lements an {A}rbitrary {N}umber of
  {Z}eros.
\newblock {\em Proc. London Math. Soc. (2)}, 15:314--321, 1916.
\newblock \href {https://doi.org/10.1112/plms/s2-15.1.314}
  {\path{doi:10.1112/plms/s2-15.1.314}}.

\bibitem[Mac04]{MacMahon1915}
Percy~A. MacMahon.
\newblock {\em Combinatory analysis. {V}ol. {I}, {II} (bound in one volume)}.
\newblock Dover Phoenix Editions. Dover Publications, Inc., Mineola, NY, 2004.
\newblock Reprint of {{\i}t An introduction to combinatory analysis} (1920) and
  {{\i}t Combinatory analysis. Vol. I, II} (1915, 1916).

\bibitem[Rio14]{Riordan}
John Riordan.
\newblock {\em An Introduction to Combinatorial Analysis}.
\newblock Princeton University Press, Princeton, 2014.
\newblock \href {https://doi.org/10.1515/9781400854332}
  {\path{doi:10.1515/9781400854332}}.

\bibitem[Rod39]{Rodrigues1839}
M.~Olinde Rodrigues.
\newblock Note sur les inversions, ou d\'erangements produits dans les
  permutations.
\newblock {\em Journal De Math\'ematiques Pures et Appliqu\'ees}, 1839.

\bibitem[Sta97]{StanEC1}
Richard~P. Stanley.
\newblock {\em Enumerative {C}ombinatorics. {V}ol. 1}, volume~49 of {\em
  Cambridge Studies in Advanced Mathematics}.
\newblock Cambridge University Press, Cambridge, 1997.
\newblock With a foreword by Gian-Carlo Rota, Corrected reprint of the 1986
  original.
\newblock \href {https://doi.org/10.1017/CBO9780511805967}
  {\path{doi:10.1017/CBO9780511805967}}.

\bibitem[Sta99]{StanEC2}
Richard~P. Stanley.
\newblock {\em Enumerative {C}ombinatorics. {V}ol. 2}, volume~62 of {\em
  Cambridge Studies in Advanced Mathematics}.
\newblock Cambridge University Press, Cambridge, 1999.
\newblock With a foreword by Gian-Carlo Rota and appendix 1 by Sergey Fomin.
\newblock \href {https://doi.org/10.1017/CBO9780511609589}
  {\path{doi:10.1017/CBO9780511609589}}.

\bibitem[SW98]{StembridgeWaugh1998}
John~R. Stembridge and Debra~J. Waugh.
\newblock A {W}eyl group generating function that ought to be better known.
\newblock {\em Indag. Math. (N.S.)}, 9(3):451--457, 1998.
\newblock \href {https://doi.org/10.1016/S0019-3577(98)80012-8}
  {\path{doi:10.1016/S0019-3577(98)80012-8}}.

\bibitem[Zab03]{Zabrocki2003}
Mike Zabrocki.
\newblock A bijective proof of an unusual symmetric group generating function,
  2003.
\newblock \href {https://doi.org/10.48550/ARXIV.MATH/0310301}
  {\path{doi:10.48550/ARXIV.MATH/0310301}}.

\end{thebibliography}

\end{document}